\documentclass[leqno,11pt]{article}

\usepackage[utf8]{inputenc}
\usepackage[T1]{fontenc}
\usepackage{cmbright}
\usepackage[english]{babel}

\usepackage{fancyhdr}

\usepackage[intlimits]{amsmath}
\usepackage{amssymb, amsfonts}

\usepackage[thmmarks, amsmath, amsthm]{ntheorem}

\usepackage{MnSymbol}
\usepackage{mathbbol}

\usepackage{bbm}

\usepackage{mathrsfs}

\usepackage[all]{xy}

\usepackage[inline]{enumitem}

\usepackage{graphicx}

\usepackage{float}
\restylefloat{table}
\usepackage{booktabs}

\usepackage{caption}

\usepackage{color}

\numberwithin{equation}{section}

\newtheorem{theoremcounter}{theoremcounter}[section]

\theoremstyle{plain}

\newtheorem{corollary}[theoremcounter]{Corollary}
\newtheorem{lemma}[theoremcounter]{Lemma}
\newtheorem{proposition}[theoremcounter]{Proposition}

\newtheorem{theorem}[theoremcounter]{Theorem}

\theoremstyle{definition}

\newtheorem{definition}[theoremcounter]{Definition}

\theoremstyle{remark}

\newtheorem{remark}[theoremcounter]{Remark}

\theoremstyle{plain}
\theoremnumbering{Alph}
\newtheorem{introtheorem}{Theorem}

\theoremnumbering{arabic}
\newtheorem{introproblem}{Problem}

\newtheorem*{introconjecture}{Conjecture}

\usepackage{xargs}                      
\usepackage[pdftex,dvipsnames]{xcolor}  
\usepackage[colorinlistoftodos,prependcaption,textsize=tiny]{todonotes}
\newcommandx{\unsure}[2][1=]{\todo[linecolor=red,backgroundcolor=red!25,bordercolor=red,#1]{#2}}
\newcommandx{\change}[2][1=]{\todo[linecolor=blue,backgroundcolor=blue!25,bordercolor=blue,#1]{#2}}
\newcommandx{\info}[2][1=]{\todo[linecolor=OliveGreen,backgroundcolor=OliveGreen!25,bordercolor=OliveGreen,#1]{#2}}
\newcommandx{\improvement}[2][1=]{\todo[linecolor=Plum,backgroundcolor=Plum!25,bordercolor=Plum,#1]{#2}}

\usepackage{tikz}
\usetikzlibrary{}
\tikzset{every picture/.style={x=1em, y=1em}}

\setlist[description]{font=\normalfont\itshape\space}

\newcommand{\cU}{\ensuremath{\mathcal{U}}}

\newcommand{\cZ}{\ensuremath{\mathcal{Z}}}

\newcommand{\rA}{\ensuremath{\mathrm{A}}}
\newcommand{\rB}{\ensuremath{\mathrm{B}}}

\newcommand{\rD}{\ensuremath{\mathrm{D}}}
\newcommand{\rE}{\ensuremath{\mathrm{E}}}

\newcommand{\rL}{\ensuremath{\mathrm{L}}}
\newcommand{\rM}{\ensuremath{\mathrm{M}}}

\newcommand{\rS}{\ensuremath{\mathrm{S}}}

\newcommand{\rU}{\ensuremath{\mathrm{U}}}
\newcommand{\rV}{\ensuremath{\mathrm{V}}}

\newcommand{\rmd}{\ensuremath{\mathrm{d}}}

\newcommand{\rmm}{\ensuremath{\mathrm{m}}}

\newcommand{\rmo}{\ensuremath{\mathrm{o}}}

\newcommand{\rmt}{\ensuremath{\mathrm{t}}}

\newcommand{\vphi}{\ensuremath{\varphi}}

\newcommand{\ol}{\overline}

\newcommand{\amid}{\ensuremath{\, | \,}}

\newcommand{\eqstop}{\ensuremath{\, \text{.}}}
\newcommand{\eqcomma}{\ensuremath{\, \text{,}}}

\newcommand{\NN}{\ensuremath{\mathbb{N}}}
\newcommand{\ZZ}{\ensuremath{\mathbb{Z}}}

\newcommand{\RR}{\ensuremath{\mathbb{R}}}
\newcommand{\CC}{\ensuremath{\mathbb{C}}}

\newcommand{\id}{\ensuremath{\mathrm{id}}}
\newcommand{\ra}{\ensuremath{\rightarrow}}

\newcommand{\hra}{\ensuremath{\hookrightarrow}}

\newcommand{\Tr}{\ensuremath{\mathrm{Tr}}}

\newcommand{\ev}{\ensuremath{\mathrm{ev}}}

\newcommand{\Aut}{\ensuremath{\mathrm{Aut}}}

\newcommand{\bs}{\ensuremath{\backslash}}

\newcommand{\ot}{\ensuremath{\otimes}}

\newcommand{\Cstar}{\ensuremath{\mathrm{C}^*}}

\newcommand{\bo}{\ensuremath{\mathcal{B}}}

\newcommand{\Cstarmax}{\ensuremath{\Cstar_\mathrm{max}}}

\newcommand{\vnt}{\ensuremath{\overline{\otimes}}}

\newcommand{\contc}{\ensuremath{\mathrm{C}_\mathrm{c}}}

\newcommand{\Linfty}{\ensuremath{{\offinterlineskip \mathrm{L} \hskip -0.3ex ^\infty}}}
\newcommand{\Ltwo}{\ensuremath{{\offinterlineskip \mathrm{L} \hskip -0.3ex ^2}}}
\newcommand{\Lone}{\ensuremath{{\offinterlineskip \mathrm{L} \hskip -0.3ex ^1}}}
\newcommand{\ltwo}{\ensuremath{\ell^2}}

\newcommand{\Ad}{\ensuremath{\mathop{\mathrm{Ad}}}}

\newcommand{\grpaction}[1]{\ensuremath{\stackrel{#1}{\curvearrowright}}}

\newcommand{\Res}{\ensuremath{\operatorname{Res}}}

\newcommand{\Fix}{\ensuremath{\mathrm{Fix}}}

\newcommand{\freegrp}[1]{\ensuremath{\mathbb{F}}_{#1}}

\newcommand{\Sym}{\ensuremath{\operatorname{Sym}}}

\newcommand{\leqc}{\ensuremath{\leq_{\mathrm{c}}}}
\newcommand{\leqo}{\ensuremath{\leq_{\mathrm{o}}}}

\newcommand{\authors}{Cyril Houdayer and Sven Raum}
\renewcommand{\title}{
  Locally compact groups acting on trees, the type ${\rm I}$ conjecture and non-amenable von Neumann algebras
}
\newcommand{\shorttitle}{Von Neumann algebras of groups acting on trees}

\begin{document}



\thispagestyle{empty}

\phantom{x}
\vspace{-5em}

\begin{center}
  \begin{minipage}[c]{0.9\linewidth}
      \begin{center}
        \LARGE
        \title
      \end{center}
      \centerline{\authors}
  \end{minipage}
\end{center}
  
\vspace{1em}

\renewcommand{\thefootnote}{}
\footnotetext{last modified on \today}
\footnotetext{Cyril Houdayer's research was supported by ERC Starting Grant GAN 637601.}
\footnotetext{Sven Raum's research leading to these results has received funding from the People Programme (Marie Curie Actions) of the European Union's Seventh Framework Programme (FP7/2007-2013) under REA grant agreement n°[622322].}

\begin{center}
  \begin{minipage}{0.8\linewidth}
    \textbf{Abstract}.
    We address the problem to characterise closed type ${\rm I}$ subgroups of the automorphism group of a tree.  Even in the well-studied case of Burger--Mozes' universal groups, non-type ${\rm I}$ criteria were unknown.  We prove that a huge class of groups acting properly on trees are not of type ${\rm I}$.  In the case of Burger-Mozes groups, this yields a complete classification of type ${\rm I}$ groups among them.  Our key novelty is the use of von Neumann algebraic techniques to prove the stronger statement that the group von Neumann algebra of the groups under consideration is non-amenable.
  \end{minipage}
\end{center}

\vspace{1em}


\section{Introduction}
\label{sec:introduction}

In discrete and topological group theory, groups acting on trees are important examples thanks to Bass-Serre theory \cite{serre80}.  In particular, the discovery of Bruhat-Tits theory \cite{serre80,bruhattits72} describing rank one reductive algebraic groups over non-Archimedean fields as groups acting on semi-regular trees provides strong motivation to study general closed subgroups of $\Aut(T)$, the automorphism group of a tree. In contrast to Bruhat-Tits buildings of higher rank \cite{weiss09}, semi-regular trees host a bigger variety of interesting groups, some of whose basic properties are not yet understood.  An intriguing problem asking us to prove surprising parallels between reductive algebraic groups and closed subgroups of $\Aut(T)$ is posed by the type ${\rm I}$ conjecture.
\begin{introconjecture}
  Let $T$ be a locally finite tree and assume that $G \leqc \Aut(T)$ is a closed subgroup acting transitively on the boundary $\partial T$.  Then $G$ is a type ${\rm I}$ group.
\end{introconjecture}
Here, a locally compact group $G$ is called a type ${\rm I}$ group if every unitary representation of $G$ generates a type ${\rm I}$ von Neumann algebra.  This is one equivalent definition of type ${\rm I}$ groups provided by \cite[Theorem 2, page 592]{glimm61}.  Bernstein and Kirillov termed ``tame'' those algebraic groups and Lie groups that are type ${\rm I}$ -- in contrast to ``wild'' groups.  In this context, type ${\rm I}$ or tameness results are derived from a positive solution to the admissibility conjecture.  The notion of type ${\rm I}$ groups bears its relevance from representation theory.  Loosely speaking, type ${\rm I}$ groups are precisely those locally compact groups all of whose unitary representations can be written as a unique direct integral of irreducible representations, thus reducing the study of arbitrary unitary representations to considerations about irreducible unitary representations.  Prominent examples of type ${\rm I}$ groups are provided by reductive algebraic groups over non-Archimedean fields \cite{bernstein74-type-I, harish-chandra70} (see also the introduction of \cite{bernsteinzelevinskii76}), adelic reductive groups \cite{clozel07}, semisimple connected Lie groups \cite[Theorem 8.1]{knapp86} and nilpotent connected Lie groups \cite[Th{\'e}or{\`e}me~1]{dixmier59}.  However, only very few results confirming the type ${\rm I}$ conjecture beyond rank one algebraic groups are known, all of them being based on combinatorial considerations for the special class of groups satisfying Tits' independence property \cite{tits70}.  See \cite{olshanskii77}, \cite{olshanskii80}, \cite{amann-thesis} and \cite{ciobotaru15}.

From the theory of algebraic groups, natural examples of non type ${\rm I}$ groups, such as most adelic nilpotent groups are known  \cite{moore65}.  For groups acting on trees the situation looks worse, since tools from Lie theory and from algebraic groups are not available in the generality of groups acting on trees.  There is one small class of groups for which non-type ${\rm I}$ results are known and it lies at the far opposite end of boundary transitive groups.  Already in the 60's Thoma proved in \cite{thoma68} that virtually abelian groups are the only discrete groups of type ${\rm I}$, which completely clarifies type ${\rm I}$ questions for discrete groups acting on trees.  In the rich spectrum between discrete groups and boundary transitive groups acting on trees, however, up to now very little is known about representation theory.  This is despite the fact that this class contains very natural examples, such as Burger-Mozes groups associated with non 2-transitive permutation groups \cite{burgermozes00-local-global}.  Astonishingly, up to now there is no result available that provides examples of non-discrete non-type ${\rm I}$ groups acting on trees.  Recent attempts to approach this problem by classical methods \cite{ciobotaru16} did not yield the desired conclusion even for the best understood examples of Burger-Mozes groups.  In this article, we take a new point of view and employ operator algebraic methods, proving that a huge class of groups acting on trees is not of type~${\rm I}$.
\begin{introtheorem}
  \label{thm:intro:non-type-I}
  Let $T$ be a locally finite tree and $G \leqc \Aut(T)$ a closed non-amenable subgroup acting minimally on $T$.  If $G$ does not act locally 2-transitive, then $G$ is a not a type ${\rm I}$ group.
\end{introtheorem}
An action of a group $G$ on a tree $T$ is called minimal, if $T$ is the smallest non-empty $G$-invariant subtree of $T$.  The action $G \grpaction{} T$ of a group on a tree is called locally 2-transitive, if for every vertex $v \in \rV(T)$ the action of the point stabiliser $G_v$ on adjacent geometric edges of $v$ is 2-transitive.  See Section~\ref{sec:groups-acting-on-trees} for more explanations.  The fact that we are able to prove a non-type ${\rm I}$ result in the generality of Theorem \ref{thm:intro:non-type-I}, insinuates the possibility of characterising those groups acting on trees that are of type~${\rm I}$.  In fact, a non-compact closed subgroup $G \leqc \Aut(T)$ is boundary transitive if and only if it is $n$-locally transitive for every $n$ in the sense of \cite{burgermozes00-local-global}: for every vertex $v$ of $T$ and every $n$ the stabiliser $G_v$ acts transitively on spheres of radius $n$ around $v$.  Since $G$ is locally 2-transitive if and only if it is 2-locally transitive, this notion provides a clear transition between groups acting not locally 2-transitively and boundary transitive groups.  We hence pose the following problem, going beyond the type ${\rm I}$ conjecture.

\begin{introproblem}
  \label{prob:type-I}
  Among closed subgroups of $\Aut(T)$, characterise those which are of type ${\rm I}$.
\end{introproblem}
\captionsetup[table]{name=Problem}
\begin{table}[h]
  \centering
  \begin{tabular}{lll}
    \toprule
    $G \leqc \Aut(T)$ is \dots & Statement & Expectation/Result \\
    \midrule
    boundary transitive & Type ${\rm I}$ conjecture & $G$ is type ${\rm I}$ \\
    \parbox[c][2.6em]{12em}{$(n-1)$-locally transitive,  \\ but not $n$-locally transitive} & open & $G$ is not type ${\rm I}$ \\
    not 2-locally transitive & Theorem \ref{thm:intro:non-type-I} & $G$ is not type ${\rm I}$ \\
    \bottomrule

  \end{tabular}
  \caption{non-amenable groups acting minimally on $T$}
\end{table}

The operator algebraic perspective introduced in this article reduces the problem to extend Theorem~\ref{thm:intro:non-type-I} to general non-boundary transitive groups to considerations in representation theory.

Burger--Mozes groups \cite{burgermozes00-local-global}, also known as universal groups acting on trees, form a particularly interesting class of examples of closed subgroups of $\Aut(T)$.  After choice of a permutation group $F \leq \rS_n$, Burger--Mozes construct groups $\rU(F)$ and index two subgroups $\rU(F)^+$ acting on the $n$-regular tree in such a way that their local action around vertices is prescribed by $F$.  These groups $\rU(F)^+$ attract particular interest of the totally disconnected group community, since they provide concrete examples of abstractly simple and compactly generated non-discrete groups \cite{burgermozes00-products, capracedemedts10, bankelderwillis14, smith15}.  Applying Theorem \ref{thm:intro:non-type-I} and combining it with known type ${\rm I}$ results \cite{amann-thesis, ciobotaru15}, we give a complete characterisation of type ${\rm I}$ groups in this important class of examples.
\begin{introtheorem}
  \label{thm:intro:characterisation-type-I-BM}
  Let $T$ be a locally finite tree and $G \leqc \Aut(T)$ a closed vertex transitive subgroup with Tits' independence property acting minimally on $\partial T$.  Then $G$ is a type ${\rm I}$ group if and only if $G$ is locally 2-transitive.

  In particular, if $F \leq \rS_n$ is a permutation group, then the Burger-Mozes groups $\rU(F)$ and $\rU(F)^+$ are type ${\rm I}$ groups if and only if $F$ is 2-transitive.
\end{introtheorem}

We prove Theorem \ref{thm:intro:non-type-I} with operator algebraic methods.  The possibility to apply operator algebraic methods to study totally disconnected groups in general and groups acting on trees  in particular has been previously suggested by the second author.  Positive results exploiting the additional flexibility provided by this idea can be found in \cite{raum15-powers} and \cite{raum15-cocompact-amenable-subgroup}.  A locally compact group is of type ${\rm I}$ if and only if its maximal group \Cstar-algebra $\Cstarmax(G)$ is a type ${\rm I}$ \Cstar-algebra in the sense of \cite{glimm61}.  Further, it is a well-known fact for operator algebraists that every type ${\rm I}$ \Cstar-algebra is amenable.  This line of thoughts suggests to study non-amenability of operator algebras associated with groups acting on trees.  Since amenability of $\Cstarmax(G)$ implies amenability of the group von Neumann algebra $\rL(G)$, Theorem \ref{thm:intro:non-type-I} is an immediate consequence of the following operator algebraic result, which is the main result of the present article.  Its proof is based on the possibility to reduce considerations about amalgamated free products of groups to plain free products of von Neumann algebras, for which clear non-amenability criteria are available.
\begin{introtheorem}
  \label{thm:intro:non-amenability}
  Let $T$ be a locally finite tree and $G \leqc \Aut(T)$ a closed non-amenable subgroup acting minimally on $T$.  If $G$ does not act locally 2-transitive, then $\rL(G)$ is non-amenable.
\end{introtheorem}

Although we want the type ${\rm I}$ conjecture to be understood as the main motivation for our present work, our von Neumann algebraic techniques allow us to prove other non-amenability criteria.  We single out the class of groups acting properly and not edge-transitively on a tree $T$, but which not necessarily embed as subgroups of $\Aut(T)$.  If $G \grpaction{} T$, we denote by $G^+ \leq G$ the subgroup of type-preserving elements, which has at most index two.
\begin{introtheorem}
  \label{thm:intro:non-amenability-2}
  Let $T$ be a tree and $G \grpaction{} T$ a proper action of a locally compact group.  Let $X = G^+ \bs T$ be the quotient graph and note that $\pi_1(X)$ is a free group.  Under either of the following sets of assumptions, $\rL(G)$ is non-amenable.
  \begin{itemize}
  \item $\operatorname{rank} \pi_1(X) \geq 2$.
  \item $\operatorname{rank} \pi_1(X) = 1$ and $G$ is non-amenable.
  \item $\pi_1(X) = 0$, $T$ is thick, $X$ is finite but not an edge and $G$ is non-amenable.
  \end{itemize}
\end{introtheorem}
A tree is called thick, if each of its vertices has valency at least three.

While in the case of a discrete group $\Gamma$, the group von Neumann algebra $\rL(\Gamma)$ is amenable if and only if $\Gamma$ is amenable, it is even an open problem to provide general non-amenability criteria for the maximal group \Cstar-algebra of a non-discrete group.  A result demonstrating the surprising difficulty of this problem is provided by Connes \cite[Corollary 7]{connes76}, who shows that the maximal group \Cstar-algebra of a connected locally compact separable group is amenable.  Only Lau-Paterson were able to provide a non-amenability criterion of general nature, although their assumption of inner amenability is very strong \cite{laupaterson91}.  Our work contributes to the understanding of further non-amenability criteria.
\begin{introtheorem}
  \label{thm:intro:non-nuclear}
  Let $T$ be a locally finite tree and $G \leqc \Aut(T)$ a closed non-amenable subgroup acting minimally on $T$.  If $G$ does not act locally 2-transitive, then $\Cstarmax(G)$ is not nuclear.
\end{introtheorem}

In line with the previous explanations and the success of operator algebraic methods applied to groups acting on trees, it is natural to pose the following problem, parallel to Problem \ref{prob:type-I}. 
\begin{introproblem}
  \label{prob:amenability}
  Characterise closed subgroups $G \leqc \Aut(T)$ for which $\rL(G)$ is amenable.  For which groups among these is $\Cstarmax(G)$ amenable?
\end{introproblem}

\section*{Acknowledgements}

The authors thank the Mittag-Leffler Institute and the organisers of the workshop ``Classification and dynamical systems II: Von Neumann Algebras'' as well as the Mathematisches Forschungsinstitut Oberwolfach and the organisers of the workshop ``\Cstar-algebras'' for providing excellent working conditions.  Part of our work on this project was completed during these workshops.  The second author thanks the University of M{\"u}nster for the excellent working conditions during the return phase of his \mbox{Marie Curie Fellowship}, when big parts of this work were done.

\section{Preliminaries}
\label{sec:preliminiaries}

In the proceeding extensive preliminaries we provide readers with either operator algebraic or group theoretic background with the necessary background to follow the main Sections \ref{sec:basic-non-amenability-results}, \ref{sec:proper-actions}, \ref{sec:proof-main-result} and \ref{sec:burger-mozes-groups}.

\subsection{Locally compact groups}
\label{sec:locally-compact-groups}

In this article we are working in the setting of topological groups and their morphisms.  This means that a homomorphism between topological groups is understood to be continuous and isomorphisms of topological groups are continuous bijective group homomorphisms with a continous inverse.

If $G$ is a locally compact group, we write $\int_G f(x) \rmd x$ for integration against a left Haar measure.  Here, the function $f$ on $G$ can take values in any Banach space, thanks to the theory of Bochner integrals.  We refer the reader to \cite{deitmarechterhoff14} for these and other basics about locally compact groups.

The following theorem characterises totally disconnected locally compact groups.  It is well-known to people working in group theory, but we give a short proof for the convenience of the reader.
\begin{theorem}[TG 39 in \cite{vandantzig36-topologische-algebra-iii}]
  Let $G$ be a locally compact group.  Then $G$ is totally disconnected if and only if its identity admits a basis of neighbourhoods consisting of compact open subgroups.
\end{theorem}
\begin{proof}
  If $G$ admits a basis of neighbourhoods consisting of compact open subgroups, then it is clear that the connected component of $e$ is $\{e\}$.  So $G$ is totally disconnected.

Assume that $G$ is totally disconnected and let $U \subset G$ be a compact open neighbourhood of the identity. We will find a compact open subgroup of $U$.  Let $\rmm: G \times G \ra G$ be the multiplication map.  Since $\{e\} \times U \subset \rmm^{-1}(U) $, for every $g \in U$ there is a neighbourhood $V_g \times U_g \subset \rmm^{-1}(U)$ of $(e, g)$.  Since $U$ is compact, we hence find identity neighbourhoods $V_1, \dots, V_n \subset G$ and open sets $U_1, \dotsc, U_n \subset G$  such that $V_i U_i \subset U$ for all $i \in \{1, \dotsc, n\}$ and $U = \bigcup_i U_i$.  Putting $V := \bigcap_i (V_i \cap V_i^{-1})$, we obtain a non-empty open symmetric set $V \subset U$ such that $V U \subset U$.  We conclude that the group $K = \bigcup_{k \in \NN} V^k \subset U$ is a compact open subgroup of $G$ lying in $U$.
\end{proof}

\begin{description}
\item[The unimodular part.]   We denote the modular function of  a locally compact group $G$ by $\Delta_G: G \ra \RR_{> 0}$.  The modular function of totally disconnected groups is nicely behaved.  If $K \leq G$ is a compact open subgroup of a locally compact group, then $\Delta_G|_K = \Delta_K \equiv 1$, shows that the kernel of $\Delta_G$ is open.  In this case, we write $G_0 := \ker \Delta_G$ for the unimodular part of $G$.
\end{description}

\subsection{Permutation groups}
\label{sec:permutation-groups}

An action of a topological group on a set is called a permutation action.  A permutation group is a group $G$ with a fixed faithful permutation action $G \grpaction{} X$.  We usually write $G \leq \Sym(X)$ for a permutation group.

If $G \grpaction{} X$ is a permutation action and $S \subset X$, we denote by $\Fix_G(S) = \{ g \in G \amid \forall s \in S: \, gs = s\}$ the pointwise stabiliser of $S$.  In case $S = \{s\}$ is a one-element set, we also write $\Fix_G(S) = G_s$.

\begin{definition}
  \label{def:two-transitivity}
  Let $G \grpaction{} X$ be a permutation action.  We say that $G$ acts 2-transitively, if $G_x \grpaction{} X \setminus \{x\}$ is transitive for every $x \in X$.
\end{definition}

\begin{remark}
  \label{rmk:two-transitivity}
  The notion of 2-transitivity for $G \grpaction{} X$ slightly defers from the usual definition.  If $|X| \geq 3$, then it is equivalent to the assumption that for each pairs $x_1 \neq x_2$ and $y_1 \neq y_2$ in $X$ there is some $g \in G$ such that $g x_i = y_i$ for $i \in \{1, 2\}$.  Only in case $|X| = 2$, our definition says that the trivial action is 2-transitive, while it does not satisfy the usual definition.

  We chose to adopt our notion of 2-transitivity to obtain clean formulations of all theorems about groups acting on trees, for which otherwise the vertices of degree two need a cumbersome separate treatment, complicating the theorems' statements.
\end{remark}

For an arbitrary topological group $G$ and an open subgroup $H \leq G$, the action $G \grpaction{} G/H$ is a permutation action.  The next lemma is a reformulation of the well-known fact that a 2-transitive permutation group is primitive.
\begin{lemma}
  \label{lem:maximal-subgroups-number-of-double-cosets}
  Let $H \leq G$ be an open subgroup of a topological group.  If $|H \bs G / H| \leq 2$, then $H \leq G$ is a maximal subgroup.
\end{lemma}
\begin{proof}
  Assume that there is a proper inclusion of open subgroups $H \leq \tilde H \leq G$ of the topological group $G$.  Then $\tilde H \subset G$ is a $H$-biinvariant set, so that $H \bs G / H = (H \bs \tilde H / H) \cup (H \bs (G \setminus \tilde H) / H)$.  Since $H \leq \tilde H$ is a proper inclusion, $|H \bs \tilde H /H| \geq 2$ and $|H \bs G / H| \geq 3$ follows.  This proves the lemma.
\end{proof}

\subsection{Groups acting on trees}
\label{sec:groups-acting-on-trees}

We follow Serre's formalism of undirected trees \cite{serre80}.  A graph $X$ is a set of vertices $\rV(X)$ with a set of (directed) edges $\rE(X)$ as well as maps $\rmo, \rmt: \rE(X) \ra \rV(X)$ and a involutive operation taking opposite edges $e \mapsto \ol{e}$ such that $\ol{e} \neq e$, $\rmo(\ol{e}) = \rmt(e)$ and $\rmt(\ol{e}) = \rmo(e)$ for all $e \in \rE(X)$.  If $X, Y$ are graphs, a graph homomorphism $\vphi: X \ra Y$ is a pair of maps $\vphi_\rV:\rV(X) \ra \rV(Y)$ and $\vphi_\rE:\rE(X) \ra \rE(Y)$ such that $\rmt_Y \circ \vphi_\rE = \vphi_\rV \circ \rmt_X$ and $\rmo_Y \circ \vphi_\rE = \vphi_\rV \circ \rmo_X$.

\begin{description}
\item[Segments and paths.]
  The standard segment of length $n$ is written $[0,n]$.  Its set of vertices is 
  \begin{equation*}
    \rV([0,n]) = \{0, \dotsc, n\}
  \end{equation*}
  and its edges are pairs 
  \begin{equation*}
    \rE([0,n]) = \{ (i, i+1) \amid i \in \{0, n-1\}\} \cup \{(i, i - 1) \amid i \in \{1, \dotsc, n\}\}
  \end{equation*}
  with $\rmo(i, j) = i$, $\rmt(i,j) = j$ and $\ol{(i,j)} = (j,i)$ for all $(i,j) \in \rE([0, n])$.  A path in a graph $X$ is graph homomorphism $[0,n] \ra X$.  We set $\rmo(s) = s(0)$ and $\rmt(s) = s(n)$ for a path $s: [0,n] \ra X$.

\item[Trees.]
  A graph $X$ is connected if there is a path between pairs of vertices in $X$.  A circuit in $X$ is the image of an injective path $s: [0, n] \ra X$ with $\rmo(s) = \rmt(s)$ for some $n \geq 1$.  A tree is a connected non-empty graph without circuits.  Let $T$ be a tree.  For $v \in \rV(T)$ we write $\rE(v) = \{e \in \rE(T) \amid \rmo(e) = v\}$ for the neighbouring edges of $v$.  We call $T$ locally finite if $\rE(v)$ is finite for all $v \in \rV(T)$.  We call $T$ thick if $\rE(v)$ contains at least three elements for all $v \in \rV(T)$.

\item[Automorphisms of a tree.]
  The group $\Aut(T)$ of graph automorphisms of a tree $T$ naturally identifies with the subgroup $\Aut(T) = \{g \in \Sym(\rV(T)) \amid v \sim w \, \Leftrightarrow \, g(v) \sim g(w)\}$ and thus inherits a totally disconnected group topology, which is uniquely defined by declaring vertex stabilisers open subgroups of $\Aut(T)$.  An action of a topological group $G$ on a tree $T$ is a continuous group homomorphism $G \ra \Aut(T)$.  If $T$ is locally finite, then vertex stabilisers are compact in $\Aut(T)$.  If $T$ is a tree and $G \grpaction{} T$ is an action, then the following statements are equivalent.
\begin{itemize}
\item $G_v \leq G$ is compact for all $v \in \rV(T)$.
\item $G \grpaction{} T$ is proper.
\end{itemize}
If further, $G \leq \Aut(T)$ embeds as a subgroup, then both previous statements are equivalent to $G \leq \Aut(T)$ being closed.

\item[Locally 2-transitive actions.]
  A group action $G \grpaction{} T$ on a tree is called locally 2-transitive if for every vertex $v \in \rV(T)$ the natural action $G_v \grpaction{} \rE(v)$ is 2-transitive.

\item[Type-preserving automorphisms.]
  An element $g \in \Aut(T)$ is called type-preserving if $2 \divides \rmd(gv, v)$ for all $v \in \rV(T)$.  Denote by $\Aut(T)^+ \leq \Aut(T)$ the group of type-preserving automorphisms.  Partitioning $\rV(T)$ is two classes by $v \sim w$ if and only if $2 \divides \rmd(v, w)$, we obtain a quotient map $\rV(T) \mapsto \{0, 1\}$.  Since $\Aut(T)$ preserves this partition, we obtain a map $\Aut(T) \mapsto \rS_2$, whose kernel is $\Aut(T)^+$.  This shows that $\Aut(T)^+ \leq \Aut(T)$ is an open subgroup of index at most two.  If $G \grpaction{} T$ is a group action on a tree, we denote by $G^+$ the inverse image of $\Aut(T)^+$ under the action map $G \ra \Aut(T)$ and call $G \grpaction{} T$ type-preserving if $G = G^+$.

  Note that if $G \grpaction{} T$ is proper, then also the type-preserving part $G^+ \leq G$ acts properly, because the restriction of a proper action to a closed subgroup remains proper.

\item[Minimal actions on trees.]
  A group action $G \grpaction{} T$ on a tree is called minimal, if $T$ is the smallest non-empty $G$-invariant subtree of $T$.

\item[Ends of a tree.]
  The standard ray $[0, \infty)$ is a tree with vertices $\rV([0, \infty)) = \NN$ and edges
\begin{equation*}
  \rE([0, \infty)) = \{(i, i+1) \amid i \in \NN\} \cup \{(i, i - 1) \amid i \in \NN_{\geq 1}\}
\end{equation*}
with $\rmo(i,j) = i$, $\rmt(i,j) = j$ and $\ol{(i,j)} = (j,i)$ for all $(i,j) \in \rE([0,\infty))$.  A geodesic ray in a tree $T$ is an injective graph homomorphism $[0, \infty) \ra T$.  Two geodesic rays are called equivalent, if after shifting they eventually agree.  Formally, $\xi \sim \xi'$ if there are $n_0 \in \NN$ and $m \in \ZZ$ such that $\xi(n + m) = \xi'(n)$ for all $n \geq n_0$.  An end of $T$ is an equivalence class of geodesic rays in $T$.

\item[Hyperbolic elements.]
  The standard two-sided geodesic $(-\infty, \infty)$ is a tree with vertices $\rV((-\infty, \infty)) = \ZZ$ and edges $\rE((-\infty, \infty)) = \{(i, i+1) \amid i \in \ZZ \} \cup \{(i , i - 1) \amid i \in \ZZ\} $.  The origin and target functions are $\rmo(i,j) = i$ and $\rmt(i, j) = j$.  The opposite edge of $(i,j) \in \rE((-\infty, \infty))$ is $\ol{(i, j)} = (j,i)$.  A (two-sided) geodesic in a tree $T$ is an injective graph homomorphism $(-\infty, \infty) \ra T$.  An element $g \in \Aut(T)$ is called hyperbolic if it neither fixes a vertex nor an edge (formally: a set $\{e, \ol{e}\} \subset \rE(T)$).  For every hyperbolic element $g \in \Aut(T)$ there is a unique two-sided geodesic $\xi$ in $T$ which is setwise fixed.  The unique $\ell \in \NN$ such that $g \circ \xi (n) = \xi(n + \ell)$ for all $n \in \ZZ$ is called the translation length of $g$.
\end{description}

The following result characterises amenable groups acting on trees.
\begin{theorem}[Theorem 1 in \cite{nebbia88}]
  \label{thm:nebbia}
  Let $T$ be a locally finite tree and $G \leqc \Aut(T)$ a closed subgroup.  Then $G$ is amenable if and only if one of the following statements holds
  \begin{itemize}
  \item $G$ fixes a vertex.
  \item $G$ stabilises an edge.
  \item $G$ fixes a point in $\partial T$.
  \item $G$ stabilises a pair of points in $\partial T$.
  \end{itemize}
\end{theorem}

\subsection{Bass-Serre theory}
\label{sec:bass-serre-theory}

Bass-Serre theory as described in \cite{serre80} (see in particular Section 5 in there) provides a natural way to study groups acting on trees $G \grpaction{} T$ by means of the quotient graph $G \bs T$ together with vertex and edge stabilisers.  The general fundamental assumption of Bass-Serre theory is that $G \grpaction{} T$ must act without inversions, i.e. if $g \in G$ fixes a geometric edge of $T$, then it fixes both its ends.  It follows from the definition that every type-preserving action satisfies this assumption.  Bass-Serre theory was originally built for discrete groups, not taking into account topologies.  Its extension to topological groups however is straight forward, as we will clarify at the end of this section.

\begin{description}
\item[Graphs of groups.]
  A graph of groups is a graph $X$ with vertex groups $(G_v)_{v \in \rV(X)}$ and edge group $(G_e)_{e \in \rE(X)}$ as well as inclusions $G_e \hra G_{\rmt(e)}$ such that $G_e = G_{\ol{e}}$.  We denote this graph of groups by $(G, X)$ for short.

\item[Fundamental group of a graph of groups.]
  If $(G, X)$ is a graph of groups, then Bass-Serre theory provides a tree $T$ --- called universal covering of $(G, X)$ --- with an action of a group $\pi_1(G, X)$ on $T$, such that $X = \pi_1(G, X) \bs T$ and $(G, X)$ is obtained by considering vertex and edge stabilisers of lifted edges from $X$ to $T$.  This construction provides a one-to-one correspondence between isomorphism classes of graphs of groups and groups acting on trees.  See Theorem 13 in \cite{serre80}.  If $G \grpaction{} T$ is a group acting on a tree with quotient graph $X = G \bs T$, we will use the convenient notation $(G, X)$ for the graph of groups obtained from this action.

\item[Contractions of subtrees.]
  See \cite[p.46ff]{serre80}. If $(G, X)$ is a graph of groups and $(G, Y)$ is a subgraph, then $\pi_1(G, Y)$ can be naturally identified with a subgroup of $\pi_1(G, X)$.  Contracting $Y \leq X$ to a vertex, we obtain a graph $X/Y$.  The contraction can be naturally turned in a graph of groups such that the vertex group of the contracted vertex $Y \in \rV(X/Y)$ is $\pi_1(G, Y)$.  We denote this graph of groups by $(G, X/Y)$.  Now we have the identity of fundamental groups $\pi_1(G, X) = \pi_1(G, X/Y)$ extending uniquely the natural inclusion of vertex and edge stabilisers of $(G,X)$ into $\pi_1(G, X/Y)$.

\item[Semi-direct product decomposition.]
  See \cite[p. 45, exercise]{serre80}.  If $(G, X)$ is a graph of groups, then the universal cover $\tilde T$ of $X$ in the usual sense can be naturally turned into a tree of groups whose vertex and edge groups are isomorphic to vertex and edge groups of $X$. We denote this tree of groups by $(G, \tilde T)$ and call it the covering tree of groups of $(G, X)$.  If $\Gamma = \pi_1(X)$ is the usual fundamental group of the graph $X$, then the action of $\Gamma$ by Deck transformations on $\tilde T$ induces an action on $\pi_1(G, \tilde T)$ and we obtain a natural isomorphism $\pi_1(G, X) \cong \pi_1(G, \tilde T) \rtimes \Gamma$.

\item[Graphs of topological groups.]  If $T$ is a tree and $G \grpaction{} T$ an action (which is understood to be continuous) of a topological group, then Bass-Serre theory naturally applies and is compatible with the topology of $G$.  Denote by $X = G \bs T$ the quotient graph and by $(G, X)$ the associated graph of groups.  In this context, vertex and edge stabilisers of $G \grpaction{} T$ are topological groups and inclusion homomorphisms are continuous and open.  Since $G$ as a topological group is uniquely determined by the abstract group $G$ together with the topology on vertex stabilisers, it makes sense to speak about graphs of topological groups.
\begin{definition}
  \label{def:graph-of-topological-groups}
  A graph of topological groups is a graph of groups $(G, X)$ with the structure of a topological group on each vertex and edge stabiliser such that inclusion homomorphisms are continuous and open.
\end{definition}
Based on Bass-Serre theory and Serre's ``d{\'e}vissage'' it is not difficult to prove that the fundamental group of a graph of topological groups carries a unique group topology turning the inclusion of vertex groups into continuous and open homomorphisms.  All previously mentioned constructions and statements remain valid in the topological setting.  For later use, we remark in particular that the semi-direct product decomposition $\pi_1(G, X) \cong \pi_1(G, \tilde T) \rtimes \pi_1(X)$ for a graph of topological groups $(G,X)$ and the covering tree of groups $(G, \tilde T)$ of $(G, X)$ gives rise to an embedding of $\pi_1(G, \tilde T)$ as an open subgroup of $\pi_1(G, x)$.  We fix the following notation: a locally compact amalgamated free product is an amalgamated free product with an open locally compact amalgam.  As previously discussed a locally compact amalgamated free product is naturally a locally compact group.
\end{description}

\subsection{Von Neumann algebras}
\label{sec:von-neumann-algebras}

Let $H$ be a complex Hilbert space and $\bo(H)$ the *-algebra of all bounded linear operators on $H$.  The topology of pointwise convergence on $\bo(H)$ is called the strong operator topology.
\begin{definition}
  \label{def:vN-algebra}
  A von Neumann algebra is a unital strongly closed *-subalgebra of $\bo(H)$ for some Hilbert space $H$.
\end{definition}

\begin{description}
\item[The $\sigma$-weak topology.]
  Since the norm topology on $\bo(H)$ is finer than the strong operator topology, every von Neumann algebra is naturally a Banach space.  By a result of Sakai \cite{sakai71}, a von Neumann algebra admits a unique isometric predual $M_*$, that is a Banach space satisfying $(M_*)^* \cong M$ isometrically.  The weak-*-topology on $M$ is called the $\sigma$-weak topology.  A positive linear map (in particular a *-homomorphism) $\vphi: M \ra N$ between von Neumann algebras is called normal if it is $\sigma$-weakly continuous.  

\item[Traces and finite von Neumann algebras.]
  A positive functional $\tau: M \ra \CC$ on a von Neumann algebra is called a trace if $\tau(xy) = \tau(yx)$ for all $x, y \in M$.  A von Neumann algebra is called finite if it admits a faithful family of normal traces, that is a family $(\tau_i)_i$ of normal traces such that $\tau_i(x^*x) = 0$ for all $i$ implies $x = 0$.

\item[Factors.]
  A factor is a von Neumann algebra $M$ with trivial centre $\cZ(M) := M \cap M' = \CC 1$.  If $M$ is an infinite dimensional factor with a non-zero trace, then $M$ is called a ${\rm II}_1$ factor.  The non-zero trace on a ${\rm II}_1$ factor is unique up to normalisation.

\item[Positive elements.] 
  We denote by $M^+ = \{ x^*x \amid x \in M\}$ the set of all positive elements in a von Neumann algebra $M$. A linear map $\vphi: M \ra N$ between von Neumann algebras is called positive if $\vphi(M^+) \subset N^+$.

\item[Conditional expectations.]
  If $N \subset M$ is an inclusion of von Neumann algebras, a conditional expectation $\rE: M \ra N$ is a projection of norm one.  It is called normal if it is $\sigma$-weakly continuous.  It satisfies $\rE(n_1 m n_2) = n_1 \rE(m) n_2$ for all $n_1, n_2 \in N$ and all $m \in M$.

\item[Weights.]
    A weight on a von Neumann algebra $M$ is an additive and positive homogeneous map $\vphi: M^+ \ra \RR_{\geq 0} \cup \{\infty\}$.  We say that $\vphi$ is faithful, if $\vphi(x) = 0$ implies $x = 0$ for every $x \in M^+$.  The weight $\vphi$ is called normal if $\sup_i \vphi(x_i) = \vphi( \sup_i x_i)$ for every bounded ascending net $(x_i)$ of positive elements in $M$.  Here $\sup_i x_i$ denotes the smallest upper bound for the net $(x_i)_i$.  One calls $\mathfrak n_\vphi = \{x \in M \amid \vphi(x^*x) < \infty\}$ the set of 2-integrable elements.  If $\vphi$ is a normal weight and $\mathfrak n_\vphi \subset M$ is $\sigma$-weakly dense, then $\vphi$ is called semifinite.  A normal faithful semifinite weight is abbreviated to an nfsf weight.

\item[Modular automorphism group.]
  If $\vphi$ is an nfsf weight on a von Neumann algebra $M$, the set $\mathfrak n_\vphi$ with the scalar product $\langle x,y \rangle := \vphi(y^*x)$  can be completed to a Hilbert space $\Ltwo(M, \vphi)$ on which $M$ is faithfully represented via left multiplication.  The map $S: x \mapsto x^*$ on $\mathfrak n_\vphi \cap \mathfrak n_\vphi^* \subset \Ltwo(M, \vphi)$ defines a conjugate linear closable unbounded operator, whose polar decomposition is denoted by $\ol{S} = J \Delta^{1/2}$.  For every $t \in \RR$, the operator $\Delta^{it}$ is a well-defined unitary on $\Ltwo(M, \vphi)$.  Tomita-Takesaki theory \cite{takesaki03-II} says that the conjugation $(\Ad \Delta^{it})_{t \in \RR}$ defines a one-parameter automorphism group of $\bo(\Ltwo(M, \vphi))$ that preserves $M$.  Its restriction to $M$ is denoted by $(\sigma^\vphi_t)_t$ and it is called the modular automorphism group of $\vphi$.
\end{description}

\subsection{Group von Neumann algebras}
\label{sec:group-vn-algebras}

We refer the reader to \cite{deitmarechterhoff14} for an introduction to locally compact groups, their representations and convolution algebras.  Let $G$ be a locally compact group and $\lambda_G: G \ra \cU(\Ltwo(G))$ its left-regular representation. It satisfies $(\lambda_G(g)f)(x) = f(g^{-1}x)$ for all $f \in \contc(G)$ and $g, x \in G$.  The group von Neumann algebra of $G$ is by definition
\begin{equation*}
  \rL(G) := \{\lambda_G(g) \amid g \in G\}'' \subset \bo(\Ltwo(G))
  \eqstop
\end{equation*}
We usually write $u_g = \lambda_G(g)$ for the canonical unitaries in $\rL(G)$.  They span an isomorphic copy of $\CC G$, to which we refer without explicitly mentioning $\lambda_G$.  Von Neumann's bicommutant theorem says that $\rL(G)$ is the strong and the $\sigma$-weak closure of the set $\CC G$.  After choice of a left Haar measure on $G$, the Bochner integral provides a natural *-homomorphism $\Lone(G) \ra \rL(G): f \mapsto \int_G f(g) \lambda_G(g) \rmd g$, which we will also denote by $\lambda_G$.  If no confusion is possible, we write $\Lone(G) \subset \rL(G)$ instead of $\lambda_G(\Lone(G)) \subset \rL(G)$.

The convolution algebra $\contc(G)$ is a left Hilbert algebra in the sense of \cite[Chapter VI.1]{takesaki03-II}.  After choice of a left Haar measure it defines a nfsf weight $\vphi$ on $\rL(G)$ that satisfies $\vphi(f) = f(e)$ for all $f \in \contc(G) \subset \rL(G)$.  This weight is called a (left) Plancherel weight of $\rL(G)$.  It satisfies $\vphi(g^* * f) = \langle f , g \rangle$ for all $f, g \in \contc(G) \subset \rL(G) \cap \Ltwo(G)$.  The modular autormorphism group of $\vphi$ satisfies $\sigma^\vphi_t(u_g) = \Delta_G^{it}(g) u_g$ for all $g \in G$.  If $G$ is a discrete group, the Plancherel weight associated with the counting measure extends to the natural normal trace $\tau: \rL(G) \ra \CC$ satisfying $\tau(u_g) = \delta_{e,g}$ for all $g \in G$.  

The next proposition is well-known and clarifies the relation between the group von Neumann algebras of a locally compact group and its closed subgroups.  It can be found for example as Theorem~A of \cite{herz72}.
\begin{proposition}
  \label{prop:closed-subgroup-vn-alg}
  Let $H \leq G$ be a closed subgroup of a locally compact group.  Then the group homomorphism $H \ni h \mapsto \lambda_G(h) \in \cU(\rL(G))$ extends to a unique injective normal *-homomorphism $\rL(H) \ra \rL(G)$.
\end{proposition}
\begin{proof}
  Denote by $\rA(G)$ Eymard's Fourier algebra \cite[Chapitre 3]{eymard64}, which is a Banach algebra densely spanned by continuous positive type functions with compact support in $G$.  By Theorem 3.10 of \cite{eymard64} we have $\rL(G)_* = \rA(G)$, i.e. there is an isomorphism $\rL(G) \cong \rA(G)^*$ carrying the $\sigma$-weak topology onto the weak-*-topology.  This isomorphism identifies $u_g \in \rL(G)$ with the evaluation functional $\ev_g \in \rA(G)^*$ for all $g \in G$.

Since $H \leqc G$ is a closed subgroup, every compactly supported function of positive type on $G$ restricts to a compactly supported function of positive type on $H$.  So Proposition 3.4 in \cite{eymard64} shows that the restriction gives rise to well-defined map $\rA(G) \ra \rA(H)$.  By Theorems 1a and 1b of \cite{herz72} (see also Theorem 4.21 of \cite{mcmullen72}), every element of $\rA(H)$ can be extended to an element of $\rA(G)$.  This shows surjectivity of the restriction map $\rA(G) \ra \rA(H)$.  It follows that the dual map $\rA(H)^* \ra \rA(G)^*$ is injective.  In view of the first paragraph this finishes the proof of the proposition.
\end{proof}

\begin{description}
\item[Averaging projections.] 
  Applied to a compact subgroup $K \leq G$ of a locally compact group, the previous proposition shows that the Bochner integral $p_K := \int_K \lambda_G(k) \rmd k \in \rL(G)$ defines a projection. Here we integrate against the Haar probability of $K$.  It is the image of $\mathbb 1_K \in \contc(K) \subset \rL(K) \subset \rL(G)$.  This projection is called averaging projection associated with $K \leq G$.
\end{description}

If $H \leqo G$ is an open subgroup, the inclusion $\rL(H) \subset \rL(G)$ from Proposition \ref{prop:closed-subgroup-vn-alg} admits a natural conditional expectation.  Also this fact is well-known.  It follows from Theorem 3.1(a) in \cite{haagerup77-dual-weights-2} in the special case $M = \CC 1$ and $\vphi = \mathbb 1_H$.  We give a short proof only for the readers convenience.
\begin{proposition}
  \label{prop:conditional-expectation-open-subgroups}
  Let $H \leq G$ be an open subgroup of a locally compact group.  Then the embedding $H \leq G$ extends to a unique injective normal *-homomorphism $\rL(H) \subset \rL(G)$.  Further, there is a unique normal conditional expectation $\rE: \rL(G) \ra \rL(H)$ satisfying $\rE(u_g) = \mathbb 1_H(g) u_g$ for all $g \in G$.
\end{proposition}
\begin{proof}
  The fact that $h \mapsto \lambda_G(h)$ extends to a unique embedding $\rL(H) \hra \rL(G)$ is the content of Proposition \ref{prop:closed-subgroup-vn-alg}.  Let us construct $\rE$.  Denote by $\vphi: \rL(G)^+ \ra [0,+\infty]$ a Plancherel weight on $\rL(G)$.  The dense subalgebra $\contc(G) \subset \rL(G)$ consists of $\vphi$-integrable elements and $\vphi(f) = f(e)$ for all $f \in \contc(G)$.  Since $H \leq G$ is open, we have $\contc(H) \subset \contc(G)$.  Further, $\contc(H) \subset \rL(H)$ is a $\sigma$-weakly dense subalgebra, implying that $\vphi$ is semifinite on $\rL(H)$.  Further, $\rL(H)$ is $\sigma^\vphi$-invariant.  By Takesaki's theorem \cite{takesaki71} there is a unique normal conditional expectation $\rE: \rL(G) \ra \rL(H)$ satisfying $\vphi(\rE(x)) = \vphi(x)$ for all $x \in \mathfrak m_\vphi$.  For $f \in \contc(H) \subset \contc(G)$ we have $\rE(f) = f$.  For $f \in \contc(G \setminus H)$ and $g \in \contc(H)$, we have $g * f \in \contc(G \setminus H)$ and $\vphi(g \rE(f)) = \vphi(g f) = 0$.  So $\rE(f) = 0$.  This shows that $\rE|_{\contc(G)}$ is the restriction map $\contc(G) \ra \contc(H)$.  If $g \in G \setminus H$, then $u_g$ is a $\sigma$-weak limit of elements in $\contc(G \setminus H)$, so that $\rE(u_g) = 0$ follows.  This proves existence of $\rE$.  Uniqueness follows from the fact that $\CC G \subset \rL(G)$ is $\sigma$-weakly dense.
\end{proof}

In case $K \unlhd G$ is a compact subgroup of a locally compact group, the group von Neumann algebras $\rL(G)$ and $\rL(G/K)$ can also be compared in a natural way.  This is the content of the next well-known proposition.
\begin{proposition}
  \label{prop:quotient-is-corner}
  Let $G$ be a locally compact group and $K \unlhd G$ a compact normal subgroup.  Then the averaging projection $p$ associated with $K$ defines a central projection in $\rL(G)$ such that $p \rL(G) \cong \rL(G/K)$.  In particular, $\rL(G)$ is non-amenable, if $\rL(G/K)$ is non-amenable.
\end{proposition}
\begin{proof}
  Recall that we can write $p = \int_K u_k \rmd k$ as a Bochner integral against the Haar probability measure of $K$.  We have 
  \begin{align*}
    \langle u_g p u_g^* \xi , \eta \rangle
    & =
    \int_K \langle u_{g k g^{-1}} \xi, \eta \rangle \rmd k \\
    & =
    \int_K \langle u_{\phantom{.}^g k} \xi, \eta \rangle \rmd k \\
    & =
    \int_K \langle u_k \xi, \eta \rangle \rmd k \\
    & =
    \langle p \xi, \eta \rangle
    \eqcomma
  \end{align*}
  for all $g \in G$ and $\xi, \eta \in \Ltwo(G)$.  The third equality follows from the fact that the Haar measure on $K$ is invariant under the conjugation action of $G$.  So $p \in \rL(G) \cap \CC G' = \cZ(\rL(G))$.

  Note that $(p \xi)(g) = \int_K \xi(kg) \rmd k$ for all $\xi \in \contc(G) \subset \Ltwo(G)$, so that $p \Ltwo(G) = \Ltwo(G)^K$ follows.  Consider the map $V: \Ltwo(G/K) \ra \Ltwo(G)$ defined by $(Vf)(g) = f(gK)$ for $f \in \contc(G/K)$.  Since $K \unlhd G$ is compact and normal, $V$ is well-defined and isometric.  A short calculation shows that $V \Ltwo(G/K) = \Ltwo(G)^K = p \Ltwo(G)$, meaning that $VV^* = p$.  So $V: \Ltwo(G/K) \ra \Ltwo(G)^K$ is a unitary.  Denoting the canonical unitaries in $\rL(G/K)$ by $v_{gK}$, $gK \in G/K$, another calculation on the dense subset $\contc(G/K)$ verifies that $p u_g V = V v_{gK}$ for every $g \in G$.  This shows $V^* p \rL(G) p V = \rL(G/K)$.  

Since $x \mapsto px$ is a conditional expectation (even a *-homomorphism) from $\rL(G)$ onto $p \rL(G) \cong \rL(G/K)$, it follows from Proposition \ref{prop:amenability-conditional-expectation} that non-amenability of $\rL(G/K)$ implies non-amenability of $\rL(G)$.
\end{proof}

\subsubsection{Hecke (von Neumann) algebras}
\label{sec:hecke-vN-algebras}

On the level of group algebras, there is a replacement for the quotient $G/K$ of a locally compact group $G$ by a compact normal subgroup $K \unlhd G$, even if we drop the assumption of normality.  This replacement is provided by Hecke algebras.

\begin{definition}
  \label{def:hecke-vN-algbera}
  Let $G$ be a totally disconnected group and $K \leq G$ a compact open subgroup.  Then $(G, K)$ is called a Hecke pair.  Let $p = p_K \in \contc(G)$  be the averaging projection associated with $K$.  Then $\contc(G,K) : = p \contc(G) p$ is called the Hecke algebra of the pair $(G,K)$ and $\rL(G, K) := p \rL(G) p$ is called the Hecke von Neumann algebra of the pair $(G,K)$.
\end{definition}

\begin{remark}
  By a result of Tzanev \cite{tzanev03} our definition of a Hecke algebra and a Hecke von Neumann algebra agree with the usual definitions.  That is, $\contc(G, K)$ is the set of all compactly supported $K$-biinvariant functions in $\contc(G)$ and $\rL(G, K)$ is the von Neumann algebra closure of $\contc(G, K)$ in its representation on $\ltwo(K \bs G)$.
\end{remark}

We will need the following formula for the dimension of a Hecke algebra in later applications.
\begin{proposition}
  \label{prop:dimension-hecke-algebra}
  Let $(G, K)$ be a Hecke pair.  Then $\dim \contc(G, K) = |K \bs G /K|$.
\end{proposition}
\begin{proof}
  We write $p = \mathbb 1_K \in \contc(G)$ for the averaging projection associated with $K \leq G$.  If $K g K \in K \bs G /K$, then $p u_g p = \mathbb 1_{KgK} \in \contc(G, K)$.  Further, it is clear that these elements generate $\contc(G, K)$ as a linear space.  Let $\vphi: \contc(G) \ra \CC$ be (the linear extension of) a Plancherel weight on $\contc(G) \subset \rL(G)$.  For $KgK \neq KhK$, we have $\vphi((p u_g p)^*pu_hp) = (\mathbb 1_K * \mathbb 1_{gK} * \mathbb 1_{h^{-1}K})(e) = 0$, since $e \notin K g K h^{-1} K$. This shows that the elements $p u_g p$ are pairwise orthogonal in $\Ltwo(G) \supset \contc(G)$.  In particular, $(p u_g p)_{KgK \in K \bs G /K}$ is a linearly independent family in $\contc(G, K)$.  This shows $\dim \contc(G, K) = | K \bs G /K|$.  
\end{proof}

\subsubsection{Group factors}
\label{sec:group-factors}

The following criterion describes discrete groups whose group von Neumann algebra is a factor.  In the well-known proof, we make use of the right-regular representation $\rho_G: G \ra \cU(\Ltwo(G))$ of a locally compact group $G$, which satisfies $(\rho_G(g)f)(x) = f(xg)$ for all $f \in \contc(G)$ and $g, x \in G$.

\begin{proposition}
  \label{prop:group-factor}
  Let $\Gamma$ be a discrete group.  Then $\rL(\Gamma)$ is a  factor if and only if every non-trivial conjugacy class in $\Gamma$ is infinite. If $\Gamma$ is non-trivial, then $\rL(\Gamma)$ is a ${\rm II}_1$ factor.
\end{proposition}
\begin{proof}
  If $\Gamma$ has a non-trivial finite conjugacy class $c \subset \Gamma$, then $x := \sum_{g \in c} u_g$ satisfies $u_g x u_g^* = x$ for all $g \in \Gamma$.  So $x \in \cZ(\rL(G))$ is a witnesses that $\rL(\Gamma)$ is not a factor.

  Assume that every conjugacy class of $\Gamma$ is infinite.  The map $\rL(\Gamma) \ni x \mapsto x \delta_e \in \ltwo(\Gamma)$ is faithful, since $x \delta_g = \rho_{g^{-1}} x \delta_e$ for all $g \in \Gamma$ and the vectors $\delta_g$, $g \in \Gamma$ are a basis of $\ltwo(\Gamma)$.  So if $x \in \cZ(\rL(G))$, it suffices to show that $x \delta_e \in \CC \delta_e$.  We have
  \begin{align*}
    (x \delta_e)(ghg^{-1})
    & =
    \langle x \delta_e, \delta_{g h g^{-1}} \rangle \\
    & =
    \langle x \delta_e , \lambda_G(g) \rho_G(g)  \delta_h \rangle \\
    & =
    \langle \lambda_G(g)^* \rho_G(g)^* x \delta_e  , \delta_h \rangle \\
    & =
    \langle   x \lambda_G(g)^* \rho_G(g)^* \delta_e ,  \delta_h \rangle \\
    & =
    \langle x \delta_e ,  \delta_h \rangle \\
    & =
    (x \delta_e)(h)
    \eqcomma
  \end{align*}
  for all $g,h \in \Gamma$.  Hence $x \delta_e$ is constant on conjugacy classes.  Since $x \delta_e$ is also 2-summable and every non-trivial conjugacy class of $\Gamma$ is infinite, it follows that $x \delta_e \in \CC \delta_e$ indeed.

If $\Gamma$ is a non-trivial icc group, then it is infinite. So $\rL(\Gamma)$ is an infinite dimensional factor.  Since $\Gamma$ is discrete, there is the natural trace on $\rL(\Gamma)$ showing that it is a ${\rm II}_1$ factor.
\end{proof}

\subsubsection{Amenable von Neumann algebras}
\label{sec:amenable-von-neumann-algebras}

A von Neumann algebra $M \subset \bo(H)$ is called injective, if there is some (not necessarily normal) conditional expectation $\rE: \bo(H) \ra M$.  Following the suggestion of Connes \cite{connes78}, we refer to this class of von Neumann algebras as amenable von Neumann algebras.

\begin{proposition}
  \label{prop:amenability-conditional-expectation}
  If $N \subset M$ is an inclusion of von Neumann algebras with conditional expectation and $M$ is amenable, then $N$ is amenable.  In particular, if $M$ is an amenable and finite von Neumann algebra, then every von Neumann subalgebra of $M$ is amenable.
\end{proposition}
\begin{proof}
  From the definitions, the first part of the proposition follows on the nose. 
  We only have to prove that every von Neumann subalgebra of a finite von Neumann algebra admits a conditional expectation.  This follows from Takesaki's theorem \cite{takesaki71} and the fact that the modular automorphism group of a trace is trivial.
\end{proof}

We fix the following important consequence of Proposition \ref{prop:amenability-conditional-expectation}.
\begin{proposition}
  \label{prop:amenability-open-subgroup}
  Let $H \leq G$ be an open subgroup of a locally compact group.  If $\rL(H)$ is non-amenable, then also $\rL(G)$ is non-amenable.
\end{proposition}
\begin{proof}
  Assume that $\rL(H)$ is non-amenable.  Proposition \ref{prop:conditional-expectation-open-subgroups} tells us that there is a natural embedding $\rL(H) \hra \rL(G)$ with a normal conditional expectation $\rL(G) \ra \rL(H)$.  We can apply Proposition \ref{prop:amenability-conditional-expectation}, in order to conclude that $\rL(G)$ is non-amenable.
\end{proof}

Let $M$ be a ${\rm II}_1$ factor, $k \in \NN_{> 0}$  $p \in \rM_k(\CC) \ot M$ a non-zero projection.  Then $p (\rM_k(\CC) \ot M)p$ is called an amplification of $M$.  Its isomophism class depends only on $t := (\Tr \ot \tau)(p)$, where $\Tr$ denotes the non-normalised trace of $\rM_k(\CC)$ and $\tau$ is the unique trace of $M$.  Hence, we write $M^t$ for this amplification.

  We also need the following simple stability result for amenable ${\rm II}_1$ factors.
\begin{proposition}
  \label{prop:amenable-amplification}
  Let $M$ be a ${\rm II}_1$ factor and $t > 0$.  Then $M$ is amenable if and only if $M^t$ is amenable.
\end{proposition}
\begin{proof}
  Fix an amenable von Neumann algebra $M \subset \bo(H)$ and a conditional expectation $\rE: \bo(H) \ra M$.  Then $\id \ot \rE: \bo(K) \vnt \bo(H) \ra \bo(K) \vnt M$ is a conditional expectation witnessing amenability of $\bo(K) \vnt M$.  If $p \in M$ is a non-zero projection and $p^\perp = 1 -p$ is its orthogonal complement, then $M \ni x \mapsto pxp \in pMp \oplus \CC p^\perp  $ is a conditional expectation.  So Proposition \ref{prop:amenability-conditional-expectation} implies amenability of $pMp \oplus \CC p^\perp$ and hence of $pMp$ .  These arguments show that every amplification of $M$ is amenable.  Further, $M = (M^t)^{1/t}$, so that the proposition follows.
\end{proof}

The next theorem is classic and a proof can be found in Theorem 2.5.8 of \cite{brownozawa08}.
\begin{theorem}
  \label{thm:amenable-group-vn-algebra-discrete-group}
  Let $\Gamma$ be a discrete group.  Then $\rL(\Gamma)$ is amenable if and only if $\Gamma$ is amenable.
\end{theorem}

\subsubsection{Free group factors and non-amenable free products of von Neumann algebras}
\label{sec:free-group-factors}

Let $M_1, M_2$ be von Neumann algebras with fixed faithful normal states $\vphi_i \in M_i^*$.  The free product von Neumann algebra $(M_1, \vphi_1) * (M_2, \vphi_2)$ is described in Chapters 1.6 and 2.5 of \cite{voiculescudykemanica92}.  It is the unique von Neumann algebra $M$ generated by isomorphic copies of $M_1$ and $M_2$ together with a normal state $\vphi$ on $M$ satisfying the freeness condition $\vphi(x_1 \dotsm x_n) = 0$ for all $x_1, \dotsc, x_n \in M_1 \cup M_2$ satisfying $x_i \in M_{j_i}$, $\vphi_{j_i}(x_i) = 0$ with $j_i \neq j_{i + 1}$ for $i \in \{1, \dotsc, n - 1\}$.  If no confusion is possible, we write $M = M_1 * M_2$ for the free product von Neumann algebra.

In this section, we briefly explain the following result due to Dykema.
\begin{theorem}[See Theorem 4.6 of \cite{dykema93-hyperfinite}]
  \label{thm:dykema-non-amenable}
  Let $M, N$ be hyperfinite tracial von Neumann algebras such that $\dim M, \dim N \geq 2$ and $\dim M + \dim N \geq 5$.  Then $M * N$ is a non-amenable von Neumann algebra.
\end{theorem}

Let $\freegrp{n}$ denote some non-abelian free group.  Then $\rL(\freegrp{n})$ is called a free group factor.  For any $k \in \NN_{> 0}$ and any non-zero projection $p \in \rM_k(\CC) \ot \rL(\freegrp{n})$, the compression $p (\rM_k(\CC) \ot \rL(\freegrp{n})) p$ is called an interpolated free group factor.  These von Neumann algebras were introduced independently in \cite{dykema94-interpolated} and \cite{radulescu94-interpolated}, where among other things it was proven that the isomorphism class of $p (\rM_k(\CC) \ot \rL(\freegrp{n})) p$ only depends on $t := \frac{n - 1}{(\Tr \ot \tau)(p)^2} + 1$, where $\Tr$ denotes the non-normalised trace on $\rM_k(\CC)$ and $\tau$ is the canonical trace on $\rL(\freegrp{n})$.  We hence write $\rL(\freegrp{t})$ for this von Neumann algebra.

\begin{proposition}
  \label{prop:interpolated-free-group-factor-non-amenable}
  Interpolated free group factors $\rL(\freegrp{t})$, $t > 1$ are non-amenable.
\end{proposition}
\begin{proof}
  Let $t > 1$ be real.  By Proposition \ref{prop:group-factor} and Theorem \ref{thm:amenable-group-vn-algebra-discrete-group} we know that $\rL(\freegrp{n})$ is a non-amenable ${\rm II}_1$ factor.  So Proposition \ref{prop:amenable-amplification} shows that $\rL(\freegrp{t}) = \rL(\freegrp{2})^{\sqrt\frac{1}{t-1}}$ is non-amenable.
\end{proof}

Now Theorem \ref{thm:dykema-non-amenable} is a consequence of the following result, which is stated explicitly in the literature.
\begin{theorem}[See Theorem 4.6 of \cite{dykema93-hyperfinite}]
  \label{thm:dykema}
  Let $M, N$ be hyperfinite tracial von Neumann algebras such that $\dim M, \dim N \geq 2$ and $\dim M + \dim N \geq 5$.  Then a direct summand of $M * N$ is isomorphic to some interpolated free group factor.
\end{theorem}

\subsubsection{Amalgamated free product von Neumann algebras}
\label{sec:amalgamated-free-product-von-neumann-algebras}

If $N \subset M$ is an inclusion of von Neumann algebras with a normal faithful conditional expectation $\rE:M \ra N$, we write $M \ominus N = \{x \in M \amid \rE(x) = 0\}$.  Given two von Neumann algebras $M_1, M_2$ with a common von Neumann subalgebra $N$ and normal faithful conditional expectations $\rE_i: M_i \ra N$, there is an amalgamated free product von Neumann algebra $(M_1, \rE_1) *_N (M_2, \rE_2)$ described in Chapter 3.8 of \cite{voiculescudykemanica92}.  It is the unique von Neumann algebra $M$ generated by isomorphic copies of $M_1$ and $M_2$ such that $M_1 \cap M_2 = N$ in $M$ as well as a normal conditional expectation $\rE:M \ra N$ obeying the freeness condition $\rE(x_1 \dotsm x_n) = 0$ for all elements $x_1, \dotsc, x_n \in M_1 \cup M_2$ with $x_i \in M_{j_i} \ominus N$ and $j_i \neq j_{i + 1}$ for all $i \in \{1, \dotsc, n - 1\}$.  Compare with Proposition 2.5 in \cite{ueda99}.

\begin{proposition}
  \label{prop:amalgamation-group-to-vn-algebra}
  Let $G = G_1 *_H G_2$ be a locally compact amalgamated free product.  Then the inclusions $\rL(G_1), \rL(G_2) \subset \rL(G)$ induce an isomorphism $\rL(G) \cong \rL(G_1) *_{\rL(H)} \rL(G_2)$ where the amalgamated free product is taken with respect to the natural conditional expectations.
\end{proposition}
\begin{proof}
Denote by $\rE: \rL(G) \ra \rL(H)$ the normal conditional expectation associated by Proposition \ref{prop:conditional-expectation-open-subgroups} with the open subgroup $H \leq G$.  It satisfies $\rE(u_g) = \mathbb 1_{H}(g) u_g$ for all $g \in G$.  Denote by $\rE_j: \rL(G) \ra \rL(G_j)$ the natural conditional expectations for $j \in \{1, 2\}$.

  We want to apply Proposition 2.5 of \cite{ueda99} to conclude the proof.  In order to do so we only need to verify the freeness condition for $\rL(G_j) \subset \rL(G)$ with respect to $\rE$.  Note that if $g_1, \dotsc, g_n \in G_1 \cup G_2$ with $g_i \in G_{j_i} \setminus H$ and $j_i \neq j_{i+1}$ for all $i \in \{j_1, \dotsc, j_{n-1}\}$, then $g_1 \dotsm g_n \in G \setminus H$.  This implies $\rE(u_{g_1} \dotsm u_{g_n} ) = 0$.

Let $x_1, \dotsc, x_n \in \rL(G_1) \cup \rL(G_2)$ with $x_i \in \rL(G_{j_i}) \ominus \rL(H)$ and $j_i \neq j_{i +1}$ for all $i \in \{1, \dotsc, n -1\}$.  Since $\CC G_j \subset \rL(G_j)$ is strongly dense for $j \in \{1, 2\}$, Kaplansky's density theorem provides us with bounded nets $(x_{\alpha, i})_\alpha$ in $\CC G_{j_i}$ for all $i \in \{1, \dotsc, n\}$ such that $x_{\alpha, i} \stackrel{\alpha \ra \infty}{\ra} x_i$ strongly.  Write $x_{\alpha, i} = \sum_{g \in G_{i_j}} c_{g, \alpha, i} u_g$.  Since $\rE_{j_i}(x_i) = 0$, we have $y_{\alpha, i} := x_{\alpha, i} - \rE_{j_i}(x_{\alpha, i}) \ra x_i$ strongly.  Since $(y_{\alpha, i})_{\alpha}$ is a bounded net, we also obtain $y_{\alpha, 1} \dotsm y_{\alpha, n} \ra x_1 \dotsm x_n$ strongly and hence also $\sigma$-weakly.  We have $y_{\alpha, i} = \sum_{g \in G_{j_i} \setminus H} c_{g, \alpha, i} u_g$, so that $\rE(y_{\alpha, 1} \dotsm y_{\alpha, n}) = 0$ for all $\alpha$ by our initial remark on $\rE$.  It follows that $\rE(x_1 \dotsm x_n) = 0$ by normality of $\rE$.  
\end{proof}

\section{Basic non-amenability results for group von Neumann algebras}
\label{sec:basic-non-amenability-results}

In this section we provide the basic non-amenability results for group von Neumann algebras, which are going to be used in Section \ref{sec:proper-actions}.  By means of Bass-Serre theory, all non-amenability questions we face, can be answered with the next Lemmas \ref{lem:non-amenability-double-cosets} and \ref{lem:non-amenability-graphs-of-groups}.

\begin{lemma}
  \label{lem:non-amenability-double-cosets}
  Let $K \leq G, H$ be two locally compact groups with a common compact open subgroup.  If $|K \bs G /K| \geq 3$ and $K \leq H$ is a proper subgroup, then $\rL(G *_K H)$ is non-amenable.
\end{lemma}
\begin{proof}
  Since $K$ is a compact open subgroup of $G$ and $H$, we have $K \leq G_0, H_0$.  So $G_0 *_K H_0 \leq G *_K H$ is an open subgroup.  So by Proposition \ref{prop:amenability-open-subgroup} it suffices to prove that $\rL(G_0 *_K H_0)$ is non-amenable.  If $|K \bs G_0 /K| \leq 2$, then $G_0$ follows compact.  Hence $G_0 \unlhd G$ is a compact open normal subgroup, showing that $G = G_0$ is unimodular.  So also $|K \bs G / K| \leq 2$, which is a contradiction.  We conclude that $|K \bs G_0 /K| \geq 3$.  Similarly, if $K = H_0$ then $H$ contains a compact open normal subgroup, and hence $H = H_0$ is unimodular.  So $K = H$, which is a contradiction.  This shows that $K \leq H_0$ is a proper inclusion.  

From now on assume that $G, H$ are unimodular groups satisfying the assumptions of the lemma.  By Proposition \ref{prop:amalgamation-group-to-vn-algebra}, there is a natural isomorphism $\rL(G *_K H) \cong \rL(G) *_{\rL(K)} \rL(H) =: M$.  Write $p = p_K \in \rL(K)$ for the averaging projection over $K$.  Let $\vphi$ be the Plancherel weight on $M$ normalised to satisfy $\vphi(p) = 1$.  Then
  \begin{equation*}
    pMp
    \supset
    p \rL(G) p *_{p \rL(K) p} p \rL(H) p
    =
    p \rL(G) p *_{\CC p} p \rL(H) p
    \eqstop
  \end{equation*}
  We have $\dim p\rL(G)p \geq | K \bs G / K| \geq 3$ by Proposition \ref{prop:dimension-hecke-algebra} and $p \rL(H) p \neq \CC p$, since $\dim p \rL(H) p \geq |K \bs H / K| \geq 2$.   Since $G, H$ are unimodular, $p \rL(G) p$ and $p \rL(H) p$ are tracial von Neumann algebras.  We can find unital hyperfinite von Neumann subalgebras $N_G \subset p \rL(G) p$ and $N_H \subset p \rL(H) p$ such that $\dim N_G \geq 3$ and $N_H \neq \CC p$.  So Dykema's Theorem \ref{thm:dykema-non-amenable} applies to show that $N_G *_{\CC p} N_H$ is non-amenable.  Since $N_G *_{\CC p} N_H$ is a non-amenable von Neumann subalgebra of the finite von Neumann algebra $p \rL(G) *_{\rL(K)}  \rL(H) p$, Proposition \ref{prop:amenability-conditional-expectation} says that also the latter is a non-amenable von Neumann algebra.  We conclude that a corner of $M$ is non-amenable, and hence $M$ is non-amenable by the same proposition.
\end{proof}

\begin{remark}
  \label{rem:use-ueda}
  Lemma \ref{lem:non-amenability-double-cosets} can be alternatively proved without reducing to the unimodular setting, if we employ Ueda's \cite{ueda10-factoriality}.  We prefer however to present a proof of Lemma \ref{lem:non-amenability-double-cosets} based on more classical theorems on free product von Neumann algebras.
\end{remark}

\begin{lemma}
  \label{lem:non-amenability-graphs-of-groups}
  Let $G$ be the fundamental group of one of the following graphs of groups $(G, X)$.
  \begin{enumerate}
  \item \label{it:segment}
    $X = 
    \begin{tikzpicture}
      \filldraw
      (0,0) circle (2pt)
      -- (3,0) circle (2pt)
      -- (6,0) circle (2pt);
    \end{tikzpicture}
    $ with compact edge groups and all inclusions proper, except for possibly one inclusion into the vertex group of the middle vertex.
  \item \label{it:ends}
    $X$ a graph with at least three terminal edges $e, f, g$ and terminal vertices $x = \rmt(e), y = \rmt(f), z = \rmt(g)$ such that $G_e, G_f, G_g$ are compact and the inclusions $G_e \hra G_x$, $G_f \hra G_y$ and $G_g \hra G_z$ are proper.
  \end{enumerate}
  Then $\rL(G)$ is non-amenable.
\end{lemma}
\begin{proof}
  Consider case \ref{it:segment} first.  The statement that $\rL(G)$ is non-amenable is equivalent to showing that $\rL(K_1 *_{L_1} K_2 *_{L_2} K_3)$ is non-amenable if $K_1, K_2, K_3$ are locally compact groups, $L_1 \leq K_1, K_2$ is a proper compact open subgroup, $L_2 \leq K_2$ is some compact open subgroup and $L_2 \leq K_3$ is a proper compact open subgroup.  We have
  \begin{equation*}
    \rL(K_1 *_{L_1} K_2 *_{L_2} K_3)
    \cong
    \bigl (\rL(K_1) *_{\rL(L_1)} \rL(K_2) \bigr ) *_{\rL(L_2)} \rL(K_3)
    \eqcomma
  \end{equation*}
  by Proposition~\ref{prop:amalgamation-group-to-vn-algebra}.  Since $L_1 \leq K_1, K_2$ is proper, the group $K_1 *_{L_1} K_2$ is non-compact.  So \mbox{$| L_2 \bs (K_1 *_{L_1} K_2) / L_2| = \infty$}.  Since also $L_2 \leq K_3$ is a proper inclusion, Lemma \ref{lem:non-amenability-double-cosets} applies to show that $\rL(G)$ is non-amenable.

  We consider case \ref{it:ends}.  Let $Y \subset X$ be the graph formed by removing the vertices $x,y,z$ and the edges $e,f,g$ from $X$.  Let $H = \pi_1(G, Y)$.  Then $G$ is  the fundamental group of the contraction $(G, Z)$ given as
  
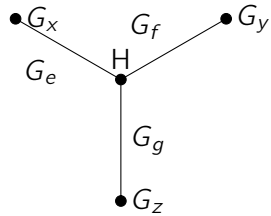
\begin{equation*}
    \begin{tikzpicture}[auto,baseline = 0]
      \filldraw (0,0) circle (2pt) node[above] {H};
      \filldraw (0,0) -- (30:4) circle (2pt) node[pos=0.5] {$G_f$} node[right] {$G_y$};
      \filldraw (0,0) -- (150:4) circle (2pt) node[pos=0.5] {$G_e$} node[right] {$G_x$};
      \filldraw (0,0) -- (0,-4) circle (2pt) node[pos=0.5] {$G_g$} node[right] {$G_z$};
    \end{tikzpicture}
  \end{equation*}
If one of the inclusions $G_e \hra H$, $G_f \hra H$ or $G_g \hra H$ is proper, the first part of the lemma applies to show that the group von Neumann algebra of an open subgroup of $G$ is non-amenable. Indeed, by symmetry we may assume that $G_e \hra H$ is proper.  Since $G_e \hra G_x$ and $G_f \hra G_y$ are proper inclusions by assumption, case \ref{it:segment} applies to $G_x *_{G_e} * H *_{G_f} G_y$, which is an open subgroup of $G$.  It follows that $\rL(G)$ is non-amenable using Proposition \ref{prop:amenability-open-subgroup}.

If $G_e = G_f = G_g = H$, then $H$ is compact and $G =  G_x *_H G_y *_H G_z$ follows from Serre's d{\'e}vissage.  The inclusions $H \hra G_x, G_y, G_z$ are all proper, so that \ref{it:segment} applies to show that $\rL(G)$ is non-amenable.
\end{proof}

\section{Groups acting properly on trees}
\label{sec:proper-actions}

In this section we consider several criteria for non-amenability of $\rL(G)$ for locally compact groups acting properly on trees.  In case $G \leq \Aut(T)$ is a subgroup of the automophisms of a locally finite tree, properness of the action is easily seen to be equivalent to closedness $G \leqc \Aut(T)$.  Our non-amenability criteria for $\rL(G)$ are organised according to the rank of the free group $\pi_1(G \bs T)$.  An increasing number of extra assumptions for $\pi_1(G \bs T)$ of lower rank is required.  For the rest of this section, we fix the setting of a proper action $G \grpaction{} T$ of a locally compact group on a tree.

Naturally, $\rL(G)$ is non-amenable, if $\pi_1(G \bs T)$ is a non-abelian free group.
\begin{proposition}
  \label{prop:rank-two-case}
  Let $T$ be a tree and $G \grpaction{} T$ a proper action of a locally compact group.  If $\operatorname{rank}\pi_1( G^+ \bs T) \geq 2$, then $\rL(G)$ is non-amenable.
\end{proposition}
\begin{proof}
  Since $G^+ \leq G$ is an open subgroup of index at most two, it suffices by Proposition \ref{prop:amenability-open-subgroup} to show that $\rL(G^+)$ is non-amenable. We may hence from now on assume that the action of $G$ on $T$ is type-preserving.

  We write $X = G \bs T$.  Let $S \subset X$ be a maximal subtree of $X$.  Denote by $(G,Y)$ the contraction of $(G, X)$ along $S$ and denote the unique vertex of $Y$ by $y$.  Then $\pi_1(Y) \cong \pi_1(X)$ is a non-abelian free group by assumption.  Let $(G, \tilde T)$ be the covering tree of groups of $Y$.  Then 
  \begin{equation*}
    G
    \cong
    \pi_1(G, X)
    \cong 
    \pi_1(G, Y)
    \cong
    \pi_1(G, \tilde T) \rtimes \pi_1(Y)
    \eqcomma
  \end{equation*}
  as described in Section \ref{sec:bass-serre-theory}.  We identify $G = \pi_1(G, \tilde T) \rtimes \pi_1(Y)$ via this natural isomorphism.

First assume that $\pi_1(G, \tilde T)$ is compact.  We denote it by $K$.  Let $p = p_K \in \rL(G)$ be the averaging projection associated with $p$.  We have $p \rL(G) p \cong \rL(G/K)$ by Proposition \ref{prop:quotient-is-corner}.  Further, $G/K \cong  \pi_1(Y)$ is a discrete non-amenable group, so that Theorem~\ref{thm:amenable-group-vn-algebra-discrete-group} shows that $\rL(G/K)$ is non-amenable.  So $\rL(G)$ has a non-amenable corner, implying that it is non-amenable itself.

Now we assume that $\pi_1(G, \tilde T)$ is non-compact.  In this case we denote it by $H$.  Since edge stabilisers of $(G, \tilde T)$ are compact and $H$ is non-compact, there is some proper inclusion of an edge group into a vertex group of $(G, \tilde T)$.  Since $(G, \tilde T)$ arises from the universal covering $\tilde T$ of $Y$, there is also some edge $e \in \rE(Y)$ such that the inclusion $G_e \leq G_y$ is non-trivial.  Since $\pi_1(Y)$ has rank at least two, there is another edge $f \in \rE(Y)$ such that $e, \ol{e}, f, \ol{f}$ are pairwise different edges in $Y$.  The subgraph of $Y$ having the vertex $y$ and the set of edges $\{e, \ol{e}, f, \ol{f}\}$ lifts to a 4-regular subtree $R$ of $\tilde T$ in which all lifts $\tilde e$ of $e$ with target $\tilde y$ define proper inclusions $G_{\tilde e} \leq G_{\tilde y}$.  We consider the following subgraph $Z$ of $R$, where the lifts of $e$ in $Z$ as well as their target vertices are marked red.
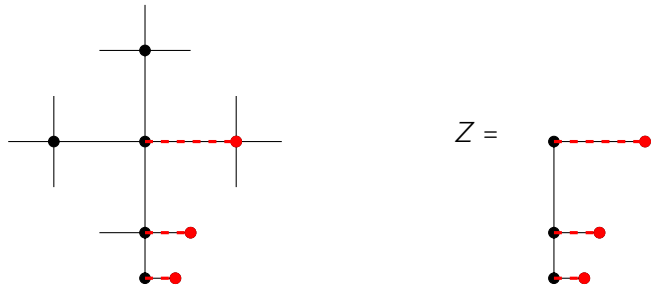
\begin{center}
  \begin{tikzpicture}[baseline=0]
    \filldraw (0,0) circle (2pt);
    \filldraw (0,0) -- (0,3) circle (2pt);
    \draw (0,3) --+(0,1.5);
    \draw (0,3) --+(-1.5,0);
    \draw (0,3) --+(1.5,0);

    \filldraw (0,0) -- (-3,0) circle (2pt);
    \draw (-3,0) --+(0,1.5);
    \draw (-3,0) --+(-1.5,0);
    \draw (-3,0) --+(0,-1.5);

    \filldraw (0,0) -- (3,0) circle (2pt);
    \draw (3,0) --+(0,1.5);
    \draw (3,0) --+(1.5,0);
    \draw (3,0) --+(0,-1.5);

    \filldraw (0,0) -- (0,-3) circle (2pt) -- (0,-4.5) circle (2pt);
    \draw (0,-3) -- (-1.5, -3);
    \filldraw (0,-3) -- (1.5, -3) circle (2pt);
    \filldraw (0,-4.5) -- (1, -4.5) circle (2pt);

    \draw[red, very thick, dashed] (0,0) -- (3,0);
    \draw[red, very thick, dashed] (0,-3) -- (1.5,-3);
    \draw[red, very thick, dashed] (0,-4.5) -- (1,-4.5);

    \filldraw[red] (3, 0) circle (2pt);
    \filldraw[red] (1.5, -3) circle (2pt);
    \filldraw[red] (1, -4.5) circle (2pt);
  \end{tikzpicture}
  \hspace{5em}
  $Z =$
  \hspace{1em}
  \begin{tikzpicture}[baseline=0]
    \filldraw (0,0) circle (2pt) -- (0,-3) circle (2pt) -- (0,-4.5) circle (2pt);
    \filldraw (0,0) -- (3,0) circle (2pt);
    \filldraw (0,-3) -- (1.5, -3) circle (2pt);
    \filldraw (0,-4.5) -- (1, -4.5) circle (2pt);

    \draw[red, very thick, dashed] (0,0) -- (3,0);
    \draw[red, very thick, dashed] (0,-3) -- (1.5,-3);
    \draw[red, very thick, dashed] (0,-4.5) -- (1,-4.5);

    \filldraw[red] (3, 0) circle (2pt);
    \filldraw[red] (1.5, -3) circle (2pt);
    \filldraw[red] (1, -4.5) circle (2pt);
  \end{tikzpicture}
\end{center}
We obtain an open subgroup $\pi_1(G, Z) \subset \pi_1(G, \tilde T)$, to which Lemma \ref{lem:non-amenability-graphs-of-groups} \ref{it:ends} applies.  So $\rL(\pi_1(G, Z))$ is non-amenable, implying that also $\rL(G)$ is non-amenable by Proposition \ref{prop:amenability-open-subgroup}.
\end{proof}

Also if $\pi_1(G \bs T)$ is a non-trivial group, we obtain a convincing criteria for non-amenability of $\rL(G)$.  In fact, non-amenability of $G$ and $\rL(G)$ are equivalent in this case.
\begin{proposition}
  \label{prop:rank-one-case}
  Let $T$ be a tree and $G \grpaction{} T$ a proper action of a non-amenable locally compact group.  If $\operatorname{rank} \pi_1(G^+ \bs T) = 1$, that is $\pi_1(G^+ \bs T) \cong \ZZ$, then $\rL(G)$ is non-amenable.
\end{proposition}
\begin{proof}
    Since $G^+ \leq G$ is an open subgroup of index at most two, it suffices by Proposition \ref{prop:amenability-open-subgroup} to show that $\rL(G^+)$ is non-amenable. We may hence from now on assume that the action of $G$ on $T$ is type-preserving.

  Write $X = G \bs T$.  We distinguish several cases.

  \textit{Case 1}. Assume that $X$ has no vertex of degree 1.  Then $X$ is a circuit.  Let $\tilde T$ be the covering tree of $X$.  It can be identified with the Cayley graph $\mathrm{Cay}(\ZZ, \{-1, 1\})$.  Since $(G, \tilde T)$ is the covering tree of a circuit, there is $p \in \NN$ such that for all $n \in \NN$
  \begin{gather*} 
    (G_{(n, n+1)} \leq G_n) = (G_{(n + p, n + 1 + p)} \leq G_{n + p}) \\
    (G_{(n, n+1)} \leq G_{n + 1}) = (G_{(n + p, n + 1 + p)} \leq G_{n + 1 + p})
    \eqstop
  \end{gather*}
  If $G_n = G_{(n, n + 1)}$ for all $n \in \ZZ$ or $G_{n + 1} = G_{(n, n + 1)}$ for all $n \in \ZZ$, then $\pi_1(G, \tilde T) = \varinjlim G_n$ is an inductive limit of compact groups.  Since $G \cong \pi_1(G, \tilde T) \rtimes \ZZ$ is non-amenable, this is a contradiction.  So there are $m,n \in \ZZ$ such that $G_{(m, m + 1)} \leq G_m$ and $G_{(n, n + 1)} \leq G_{n + 1}$ are proper inclusions.  Shifting indices by $p$, we may find $m < n < o \in \ZZ$ such that
  \begin{gather*}
    G_{(m, m + 1)} \leq G_m \quad \text{ is proper,} \\
    G_{(n, n + 1)} \leq G_{n + 1} \quad \text{ is proper,} \\
    G_{(o, o + 1)} \leq G_{o + 1} \quad \text{ is proper.}
  \end{gather*}
  Fixing $m, n \in \ZZ$ with such properties, we can assume $o > n$ to be minimal with these properties.  Let
  \begin{equation*}
    H := \langle G_{(m, m+1)}, G_{m+1}, G_{m + 2}, \dotsc, G_n, G_{(n, n+1)} \rangle \eqstop
  \end{equation*}
  Further note that $\langle G_{(n+1, n+2)}, G_{n+2}, G_{n + 3}, \dotsc, G_o, G_{(o, o + 1)} \rangle = G_{(n + 1, n + 2)}$ by minimality of $o$.  We obtain that
  \begin{align*}
    \langle G_m, G_{m + 1}, \dotsc, G_{o+1} \rangle
    & \cong
    G_m *_{G_{(m, m+1)}} G_{m + 1} *_{G_{(m+1, m+2)}} \dotsm *_{G_{(o,o+1)}} G_{o + 1} \\
    & \cong
    G_m *_{G_{(m, m+1)}} H *_{G_{(n, n+1)}} G_{n+1} *_{G_{(n+1, n+2)} }\dotsm *_{G_{(o,o+1)}} G_{o + 1} \\
    & \cong
    G_m *_{G_{(m, m + 1)}} H *_{G_{(n, n+1)}} G_{n+1} *_{G_{(o,o+1)}} G_{o+1}
    \eqstop
  \end{align*}
  This is an open subgroup of $G$.  If either $H \neq G_{(n, n + 1)}$ or $H \neq G_{(m, m + 1)}$, then Lemma \ref{lem:non-amenability-double-cosets} applies to $G_m *_{G_{(m, m + 1)}} H *_{G_{(n, n + 1)}} G_{n + 1}$ and shows that its group von Neumann algebra is non-amenable.  So also $\rL(G)$ is non-amenable by Proposition \ref{prop:amenability-open-subgroup}.  In case $G_{(n, n + 1)} = H = G_{(m, m + 1)}$, we have
  \begin{equation*}
    G_m *_{G_{(m, m + 1)}} H *_{G_{(n, n+1)}} G_{n+1} *_{G_{(o,o+1)}} G_{o+1}
    \cong
    G_m *_H G_{n+1} *_{G_{(o,o+1)}} G_{o+1}
  \end{equation*}
  and $H \leq G_m, G_{n + 1}$ as well as $G_{(o,o+1)} \leq G_{o + 1}$ are proper inclusions.  So Lemma \ref{lem:non-amenability-graphs-of-groups} \ref{it:segment} applies to show that the group von Neumann algebra of $G_m *_H G_{n+1} *_{G_{(o, o+1)}} G_{o+1}$ is non-amenable.

\textit{Case 2}. Assume that $X$ has some vertex of degree 1.  Let $v \in \rV(X)$ have degree 1 and let $e \in \rE(X)$ be the unique edge satisfying $\rmt(e) = v$.  If $G_e = G_v$, then any lift of $v$ to $T$ is a terminal vertex.  We may hence remove $v$ and $e$ from $X$ without changing $G$.  This either reduces to Case 1, or it provides us with a vertex $v \in \rV(X)$ of degree 1 and an edge $e \in \rE(X)$ with $\rmt(e) = v$ such that $G_e \leq G_v$ is a proper inclusion.  Let $(G, \tilde T)$ be the  covering tree of groups of $(G, X)$.  Then $(G, \tilde T)$ takes the form

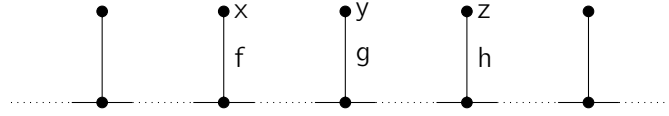
\begin{center}
  \begin{tikzpicture}
    
    \draw[dotted] (-7,0) -- (-5,0);
    \filldraw (-5,0) -- (-4,0) circle (2pt);
    \filldraw (-4,0) -- ++(0,3) circle (2pt);
    \draw (-4,0) -- (-3,0);
    \draw[dotted] (-3,0) -- (-1,0);
    \draw (-1,0) -- (0,0);

    \filldraw (0,0) circle (2pt) -- ++(0,3) circle (2pt)
    node[pos=0.5, right] {f} node[right] {x};
    \draw (0,0) -- (1,0);
    \draw[dotted] (1,0) -- (3,0);
    \draw (3,0) -- (4,0);
    \filldraw (4,0) circle (2pt) -- ++(0,3) circle (2pt)
    node[pos=0.5, right] {g} node[right] {y};
    \draw (4,0) -- (5,0);
    \draw[dotted] (5,0) -- (7,0);
    \draw (7,0) -- (8,0);
    \filldraw (8,0) circle (2pt) -- ++(0,3) circle (2pt)
    node[pos=0.5, right] {h} node[right] {z};
    
    \draw (8,0) -- (9,0);
    \draw[dotted] (9,0) -- (11,0);
    \filldraw (11,0) -- (12,0) circle (2pt);
    \filldraw (12,0) -- ++(0,3) circle (2pt);
    \draw (12,0) -- (13,0);
    \draw[dotted] (13,0) -- (15,0);    
  \end{tikzpicture}
\end{center}
where $x,y,z \in \rV(\tilde T)$ are lifts of $v$ and $f,g,h \in \rE(\tilde T)$ are lifts of $e$.  The inclusions $G_f \leq G_x$, $G_g \leq G_y$ and $G_g \leq G_z$ are proper, since they are isomorphic with $G_e \leq G_v$.  So Lemma \ref{lem:non-amenability-double-cosets} \ref{it:ends} applies and says that $\pi_1(G, \tilde T)$ has a non-amenable group von Neumann algebra.  Since $\pi_1(G, \tilde T) \leq \pi_1(G, \tilde T) \rtimes \ZZ \cong G$ is an open subgroup, Proposition \ref{prop:amenability-open-subgroup} implies that $\rL(G)$ is non-amenable.
\end{proof}

As can be expected, the case $\pi_1(G \bs T) = 0$ becomes the most subtle.  This is due to the fact that there are many edge transitive closed type ${\rm I}$ subgroups of $\Aut(T)$.  Their group von Neumann algebras are in particular amenable.  We obtain a non-amenability result in this case, which is sufficient for all applications presented in this article.
\begin{proposition}
  \label{prop:rank-zero-case}
  Let $T$ be a thick tree and $G \grpaction{} T$ a proper action of a non-amenable locally compact group such that $G^+ \bs T$ is finite and satisfies $\pi_1(G^+ \bs T) = 0$.  If $G^+ \bs T$ is not an edge, then $\rL(G)$ is non-amenable.
\end{proposition}
\begin{proof}
  Since $G^+ \leq G$ is an open subgroup of index at most two, it suffices by Proposition \ref{prop:amenability-open-subgroup} to show that $\rL(G^+)$ is non-amenable. We may hence from now on assume that the action of $G$ on $T$ is type-preserving.

  We write $X = G \bs T$.  Let $v \in \rV(X)$ be a terminal vertex of $X$ and $\tilde v \in \rV(T)$ a lift of $v$.  If $e \in \rE(X)$ is the unique edge satisfying $\rmt(e) = v$, then $|G_v/G_e| = |\rE(\tilde v)| \geq 3$, since $T$ is thick. In particular $G_e \leq G_v$ is a proper inclusion.  So if $X$ has at least three terminal edges, then Lemma \ref{lem:non-amenability-graphs-of-groups} \ref{it:ends} applies to show that $G = \pi_1(G, X)$ has a non-amenable group von Neumann algebra.  Otherwise, $X$ is a finite segment, which we can identify with the standard segment $[0,n]$ for some $n \in \NN_{>0}$.  Since $G$ does not act edge transitively, we have $n \geq 2$.  We distinguish different cases.

\textit{Case 1}. We have a proper inclusion $G_{(0,1)} \leq G_1$ or $G_{(n-1, n)} \leq G_{n-1}$.  By symmetry we may assume that $G_{(n-1,n)} \leq G_{n-1}$ is a proper inclusion.  Put
\begin{equation*}
  H = G_1 *_{G_{(1,2)}}  \dotsm *_{G_{(n-2, n-1)}} G_{n-1}
  \eqstop
\end{equation*}
Then $G = G_0 *_{G_{(0,1)}} H *_{G_{(n-1, n)}} G_n$ with $G_{(0,1)}, G_{(n-1, n)}$ compact and with proper inclusions $G_{(0,1)} \leq G_0$ and $G_{(n-1, n)} \leq H$ as well as $G_{(n-1, n)} \leq G_n$.  So Lemma \ref{lem:non-amenability-graphs-of-groups} \ref{it:segment} applies to show that $\rL(G)$ is non-amenable.

\textit{Case 2}. We have $G_{(0,1)} = G_1$ or $G_{(n-1, n)} = G_{n-1}$.  By symmetry we may assume that $G_{(0,1)} = G_1$.  Let $k \in \NN$ be maximal with the property that
\begin{equation*}
  G_0 \geq G_{(0,1)} = G_1 \geq G_{(1,2)} = G_2 \geq \dotsm \geq G_{(k-1, k)} = G_k
  \eqstop
\end{equation*}
We know that $G_{(n-1, n)} \leq G_n$ is a proper inclusion, implying that $k \leq n-1$.  So
\begin{equation*}
  G
  \cong
  G_0 *_{G_{(0,1)}} G_1 *_{G_{(1,2)}} \dotsm *_{G_{(n-1, n)}} G_n
  \cong
  (G_0 *_{G_{(k, k+1)}} G_{k+1}) *_{G_{(k+1, k+2)}} \dotsm *_{G_{(n-1,n)}} G_n
  \eqstop
\end{equation*}
We will show that the open subgroup $G_0 *_{G_{(k, k+1)}} G_{k+1} \leq G$ has a non-amenable group von Neumann algebra.  Thanks to Proposition \ref{prop:amenability-open-subgroup}, this will finish the proof.

Let $i \in \{1, \dotsc, k\}$.  If $v \in \rV(T)$ is a lift of $i \in \rV([0,n])$ and $e,f \in \rE(v)$ are lifts of $(i-1, i), (i, i+1) \in \rE(k)$, respectively, then
\begin{equation*}
  \rE(v)
  \cong
  G_v/G_e \sqcup G_v/G_f
  =
  \{G_e\} \sqcup G_v/G_f
  \eqstop
\end{equation*}
Since $|\rE(v)| \geq 3$, it follows that $|G_i/G_{(i, i+1)}| = |G_v/G_f| \geq 2$.  So $G_i \geq G_{(i, i+1)}$ is a proper inclusion for all $i \in \{1, \dotsc, k\}$.  Since also $G_0 \geq G_{(0,1)}$ is a proper inclusion, we have the chain of proper inclusions 
\begin{equation*}
  G_0 \gneq G_{(0,1)} = G_1 \gneq G_{(1,2)}
  \eqstop
\end{equation*}
This shows that $G_{(k, k+1)} \leq G_0$ is not a maximal subgroup.  So Lemma \ref{lem:maximal-subgroups-number-of-double-cosets} shows that $|G_{(k,k+1)} \bs G_0 / G_{(k , k+1)}| \geq 3$.  We checked all conditions to apply Lemma \ref{lem:non-amenability-double-cosets} to $G_0 *_{G_{(k, k + 1)}} G_{k + 1}$, finishing the proof of the proposition.
\end{proof}

We end this section, by a non-amenability result for edge transitive groups $G \grpaction{} T$.  A condition on the local action of $G \grpaction{} T$ around a vertex ensures non-amenability of $\rL(G)$.
\begin{proposition}
  \label{prop:edge-transitive-case}
  Let $T$ be a thick tree and $G \grpaction{} T$ a proper action of a locally compact group.  Assume that $G^+$ is edge transitive but not locally 2-transitive.  Then $\rL(G)$ is non-amenable.
\end{proposition}
\begin{proof}
  Consider the open subgroup $G^+ \leq G$ of index at most two.  Note that $G^+ \grpaction{} T$ is not locally 2-transitive, since $G \grpaction{} T$ is not locally 2-transitive.  By Proposition \ref{prop:amenability-open-subgroup} it hence suffices to show that $\rL(G^+)$ is non-amenable. We may hence from now on assume that the action of $G$ on $T$ is type-preserving and $G \bs T$ is an edge.

  Since $G$ is not locally 2-transitive, there is some $v \in \rV(T)$ such that $G_v \grpaction{} \rE(v)$ is not 2-transitive.  Let $e \in \rE(v)$ and $w = \rmt(e)$.  Bass-Serre theory says that $G \cong G_v *_{G_e} G_w$, since $G$ is edge transitive and type-preserving.  Since $G_v \grpaction{} \rE(v)$ is transitive, we have a $G_v$ equivariant identification $\rE(v) \cong G_v/G_e$.  Since $G_v \grpaction{} \rE(v)$ is not 2-transitive, we further have $|G_e \bs G_v / G_e| = |G_e \bs \rE(v)| = 1 + |G_e \bs (\rE(v) \setminus \{e\})| \geq 3$.  Note also that $G_e \leq G_w$ is a proper inclusion, since $|G_w/G_e| = |\rE(w)| \geq 3$.  Now Lemma~\ref{lem:non-amenability-double-cosets} applies to show that $\rL(G)$ is non-amenable.
\end{proof}

\section{Proof of Theorems \ref{thm:intro:non-amenability} and \ref{thm:intro:non-amenability-2}}
\label{sec:proof-main-result}

To start this section let us note that Theorem \ref{thm:intro:non-amenability-2} simply summarises Propositions \ref{prop:rank-two-case}, \ref{prop:rank-one-case} and \ref{prop:rank-zero-case}.  We will thus devote the rest of this section to the proof of Theorem \ref{thm:intro:non-amenability}.

\begin{lemma}
  \label{lem:reduce-infinitely-many-ends}
  Let $T$ be a locally finite tree such that $\Aut(T)$ is not virtually abelian and acts minimally on $T$.  Then $T$ has infinitely many ends.
\end{lemma}
\begin{proof}
  We assume that $T$ has only finitely many ends and deduce a contradiction.  If $T$ has no end, then it is finite and $\Aut(T)$ is a finite group, hence virtually abelian.  So $T$ has at least one end.  If $T$ has exactly one end, then it contains a unique maximal geodesic ray.  This ray is pointwise fixed by $\Aut(T)$, which contradicts minimality of $\Aut(T) \grpaction{} T$.  If $T$ has exactly 2 ends, then $\Aut(T)$ setwise fixes the unique two-sided infinite geodesic of $T$.  By minimality of $\Aut(T) \grpaction{} T$, it follows that $T \cong \mathrm{Cay}(\ZZ, \{-1, 1\})$.  Then $\Aut(T) \cong \rD_\infty$ is a dihedral group, which is virtually abelian.  This shows that $T$ has at least 3 ends.  Let $F = \{(x,y) \cap (y,z) \cap (z,x) \amid x,y,z \text{ pairwise different ends of } T\}$.  Since $T$ has only finitely many ends, $F$ is finite.   Further, its definition makes it clear that $F$ is $\Aut(T)$-invariant, contradicting minimality of $\Aut(T) \grpaction{} T$.  This finishes the proof of the lemma.
\end{proof}

\begin{lemma}
  \label{lem:reduction-thick-tree}
  Let $T$ be a tree with at least some vertex of degree 3 and such that $\Aut(T)$ acts minimally on $T$.  Then there is a thick tree $S$ such that
  \begin{itemize}
  \item $\rV(S) \subset \rV(T)$,
  \item $\rV(S)$ is $\Aut(T)$-invariant, and
  \item the restriction map $\Aut(T) \ra \Sym(\rV(S))$ induces an isomorphism of topological groups $\Aut(T) \cong \Aut(S)$.
  \end{itemize}
  Further,
  \begin{itemize}
  \item if $G \leqc \Aut(T)$ is a closed subgroup acting minimally on $T$, then also $G \grpaction{} S$ is minimal and,
  \item $G$ acts locally 2-transitively on $T$ if and only if it acts locally 2-transitively on $S$.
  \end{itemize}
\end{lemma}
\begin{proof}
  We define
  \begin{gather*}
    \rV(S) = \{v \in \rV(T) \amid \deg(v) \geq 3\} \quad \text{ and } \\
    \rE(S) = \{s: [0,n] \hra T \amid n \geq 1, \, \deg(s(0)), \deg(s(n)) \geq 3, \, \forall i \in \{1, \dotsc, n-1\}: \deg(s(i)) = 2 \}
    \eqcomma
  \end{gather*}
  with origin $\rmo(s) = s(0)$ and target $\rmt(s) = s(n)$.  It is clear that $S$ is a non-empty graph.

  We first show that $S$ is a tree. To this end we prove that $S$ is connected and that every circuit in $S$ has backtracking.  Let $v, w \in \rV(S)$.  There is some injective path $s:[0,n] \hra T$ such that $s(0) = v$ and $s(n) = w$.  Let $B = \{i \in \{1, \dotsc, n-1\} \amid \deg(s(i)) \geq 3\}$.  If $B = \emptyset$, then $s \in \rE(S)$ is an edge between $v$ and $w$.  Otherwise let $i_1 < \dotsm < i_k$ be an enumeration of $B$.   Put $i_0 := 0$ and $i_{k+1} := n$.   Set $s_j := s|_{[i_j, i_{j+1}]}$ for $j \in \{0, \dotsc, k\}$.  Then $s_j \in \rE(S)$ (after identifying $[i_j, i_{j+1}] \cong [0, i_{j+1} - i_j]$) for all $j \in \{0, \dotsc, k\}$.  We have
  \begin{center}
    \vspace{-1em}
    \begin{minipage}{0.6\textwidth}
      \begin{gather*}
        \rmo(s_0) = s_0(0) = s(0) = v \eqcomma\\
        \rmt(s_k) = s_k(n) = s(n) = w \eqcomma \\
        \rmt(s_j) = s_j(i_{j+1}) = s(i_{j+1}) = s_{j+1}(i_{j+1}) = \rmo(s_{j+1}) \eqcomma
      \end{gather*}
    \end{minipage}
    \begin{minipage}{0.2\textwidth}
      \begin{align*}
        & \\
        & \text{and} \\
        & \text{for all all } j \in \{0, \dotsc, k-1\} \eqstop
      \end{align*}
    \end{minipage}
 \end{center}
  This shows that $s_0, \dotsc, s_k$ define a chain of edges connecting $v$ and $w$ in $S$.  So $S$ is connected.

  Let now $s_0, \dotsc, s_k \in \rE(S)$, $s_j:[i_j, i_{j+1}] \hra T$ define a circuit in $S$.  Define $s:[0, i_{k+1}] \ra T$ as the path that agrees with $s_j$ on $[i_j, i_{j+1}]$.  Then $s$ is a circuit in $T$ since,
  \begin{equation*}
    \rmo(s) = s_0(0) = \rmo(s_0) = \rmt(s_k) = s_k(i_{k+1}) = \rmt(s)
    \eqstop
  \end{equation*}
  Since $T$ is a tree, there is some $l \in [0, i_{k+1} - 2]$ such that $s((l, l + 1)) = \ol{s((l + 1, l + 2))}$.  Since $s_j$ is an injective path for all $j \in \{0, \dotsc, k\}$, we must have $l = i_j - 1$ for some $j \in \{1, \dotsc, k\}$.  This means that $s_j$ and $s_{j+1}$ are injective paths in $T$ all of whose non-terminal vertices have degree 2 and such that the last edge of $s_j$ is the conjugate of the first edge of $s_{j+1}$ (i.e. $s_j((i_j -1, i_j)) = \ol{s_{j+1}((i_j, i_j + 1))}$).  This implies $s_j = \ol{s_{j+1}}$.  So $s_0, \dotsc, s_k$ has backtracking and we conclude that $S$ is a tree.

  If $g \in \Aut(T)$, $v \in \rV(T)$, then $\deg(gv)  = \deg(v)$, so that $g \rV(S) = \rV(S)$ follows.  Denote by $\Res: \Aut(T) \ra \Sym(\rV(S))$ the restriction homomorphism.  We show that $\Res(\Aut(T)) \subset \Aut(S)$.  Assume $v, w \in \rV(S)$ are adjacent in $S$.  Then there is $s \in \rE(S)$, $s: [0,n] \hra T$ such that $s(0) = v$, $s(n) = w$.  Since $gs \in \rE(S)$, with $(gs)(0) = g(s(0)) = gv$ and $(gs)(n) = g(s(n)) = gw$, we also have $gv \sim gw$ in $S$.  This shows that $\Res(g)$ is a bijective graph homomorphism.  Since $\Res(g^{-1}) = \Res(g)^{-1}$, it follows that $\Res(g) \in \Aut(S)$.

  We show that $\Res$ is injective.  Assume that $\Res(g) = \id_S$ for some $g \in \Aut(T)$.  Then $g|_{\rV(S)} = \id_{\rV(S)}$.  Since $\rV(S)$ is $\Aut(T)$-invariant and $\Aut(T) \grpaction{} T$ is minimal by assumption, $T$ is the convex closure of $\rV(S)$.  So $g = \id_T$.

  We show that $\Res$ is surjective.  Let $h \in \Aut(S)$.  We want to define $\beta(h)(s(i)) := (hs)(i)$ for all $s \in \rE(S)$, $s : [0,n] \hra T$ and $i \in \{0, \dotsc, n\}$.  We first prove that this gives rise to a well-defined map $\beta(h): \rV(T) \ra \rV(T)$.  If $v \in \rV(S) \subset \rV(T)$ and $s \in \rE(S)$, $s: [0,n] \ra T$ satisfies $s(i) = v$ for some $i \in \{0,1, \dotsc, n\}$, then $i \in \{0,n\}$.  We obtain
  \begin{align*}
    (hs)(0) = \rmo(hs) = h(\rmo(s)) = hv & \phantom{,} \quad \text{ or} \\
    (hs)(n) = \rmt(hs) = h(\rmt(s)) = hv & \eqcomma \quad \text{respectively.}
  \end{align*}
  So the image $\beta(h)v = hv$ is independent of the choice of $s$.  If $v \in \rV(T) \setminus \rV(S)$, then there is some $s \in \rE(S)$ containing $v$, that is, writing $s:[0,n] \hra T$, there is some $i \in \{1, \dotsc, n - 1\}$ such that $v = s(i)$.  This follows from the fact that the convex closure of $\rV(S)$ is $T$.  If $s, s' \in \rE(S)$ both contain $v$, then $s' = s$ or $s' = \ol{s}$.  Take $s \in \rE(S)$, $s: [0,n] \hra T$ and $i \in \{1, \dotsc, n - 1\}$ such that $s(i) = v$.  Then
  \begin{equation*}
    (h \ol{s})(n - i) = \ol{(hs)}(n - i) = (hs)(i)
  \end{equation*}
  showing that $\beta(h)v$ is well-defined. Note that $\beta(h^{-1}) = \beta(h)^{-1}$, so that we obtain a map $\beta: \Aut(S) \ra \Sym(\rV(T))$.  Also note that $\beta$ is a group homomorphism.  If $h \in \Aut(S)$ and $v, w \in \rV(T)$ are adjacent, then there is $s \in \rE(S)$, $s:[0,n] \hra T$ and $i \in \{0, \dotsc, n-1\}$ such that $s(i) = v$, $s(i+1) = w$.  Then
  \begin{equation*}
    \beta(h)v = (hs)(i) \sim (hs)(i+1) = \beta(h)w
    \eqstop
  \end{equation*}
  This shows that $\beta(h)$ is a graph homomorphism.  Since also $\beta(h)^{-1} =\beta(h^{-1})$ is a graph homomorphism, we find that $\beta: \Aut(S) \ra \Aut(T)$.  It is clear that $\beta(h)|_{\rV(S)} = h$ for all $h \in \Aut(S)$, which shows that $\Res \circ \beta = \id_{\Aut(S)}$.  So $\Res$ is surjective.

  By construction $\Res$ is a continuous map.  Also $\beta$ is continuous as it can be easily checked on a neighbourhood basis of the identity in $\Aut(S)$.  So $\Res$ is an isomorphism of topological group.

  Now assume that $G \leqc \Aut(T)$ is a group acting minimally on $T$.  We will show that $G \grpaction{} S$ is also minimal.  To this end, take $v \in \rV(S)$ and $g \in G$.  Then $\Res(g)v = gv$.  Further, the construction of $S$ shows that $[v, gv]_T \cap \rV(S) = [v, gv]_S \cap \rV(S)$.  So a vertex of $\rV(S)$ is in the convex closure of $Gv$ inside $T$ if and only if it is in the convex closure of $\Res(G)v$ inside $S$.  This suffices to conclude that $G \grpaction{} S$ is minimal.

  Now let us consider local 2-transitivity of $G \grpaction{} T$ and $G \grpaction{} S$.  For every $v \in \rV(S)$ the map $\rE_S(v) \ni s \mapsto s((0, 1)) \in \rE_T(v)$ is a bijection.  Let us denote its inverse by $\Res_v: \rE_T(v) \ra \rE_S(v)$.  Then $\Res(g)\Res_v(e) = \Res_{gv}(ge)$ for all $g \in \Aut(T)$, $v \in \rV(T)$ and $e \in \rE_T(v)$.  Together with the observation that $\Res(G)_v = \Res(G_v)$, this directly implies that $G \grpaction{} T$ is locally 2-transitive if and only if $G \grpaction{} S$ is locally 2-transitive.  This finishes the proof of the lemma.
\end{proof}

The following lemma is well-known.  Its proof can be found for example as Lemma 2.4 in \cite{capracedemedts10}
\begin{lemma}
  \label{lem:compactly-generated-cocompact}
  Let $T$ be a locally finite tree and let $G \leqc \Aut(T)$ be a closed subgroup acting minimally on $T$.  Then $G$ is compactly generated if and only if $G$ acts cocompactly on $T$.
\end{lemma}

We now come to the major reduction result necessary to apply results from Section \ref{sec:proper-actions} in the proof of Theorem~\ref{thm:intro:non-amenability}.
\begin{proposition}
  \label{prop:reduction}
  Let $T$ be a locally finite tree with infinitely many ends.  Let $G \leqc \Aut(T)$ be a closed non-amenable subgroup acting minimally on $T$.  Then there is an open non-amenable subgroup $H \leq G$, a compact normal subgroup $K \unlhd H$ and a locally finite thick tree $S$ such that $H/K \leqc \Aut(S)$ acts minimally, cocompactly and in a type-preserving way on $S$.

If $G \grpaction{} T$ is not locally 2-transitive, then also $H/K \grpaction{} S$ can be chosen to be not locally 2-transitive. 
\end{proposition}
\begin{proof}
  Since $T$ is a locally finite tree with infinitely many ends and $G \leqc \Aut(T)$ is a non-amenable subgroup acting minimally on $T$, it contains a hyperbolic element and $T$ is the convex closure of all translation axes of hyperbolic elements in $G$.  Let $H \leq G$ be an open non-amenable compactly generated subgroup.  In case $G \grpaction{} T$ is not locally 2-transitive, there is $v \in \rV(T)$ with $\deg(v) \geq 3$ and $G_v \grpaction{} \rE(v)$ is not 2-transitive.  Since $G \grpaction{} T$ is minimal, the convex closure of translation axes of hyperbolic elements in $G$ equals $T$.  Adding finitely many elements to a compact generating set of $H$, we may hence assume that $\rE(v)$ lies in the convex closure $T'$ of all translation axes of hyperbolic elements in $H$.  Note that $H \grpaction{} T'$ is minimal by construction.  The fix group $K = \Fix_G(T') \cap H$ is compact and normal in $H$.  It is the kernel of the map $H \ra \Aut(T')$.  We obtain the closed subgroup $H/K \leqc \Aut(T')$.  Since $G_v \grpaction{} \rE(v)$ is not 2-transitive, also $H_v \grpaction{} \rE(v)$ is not 2-transitive.  Further, this action factors through $(H/K)_v$, since $\rE(v) \subset \rE(T')$.  We thus find that $H/K \grpaction{} T'$ is not locally 2-transitive in case $G \grpaction{} T$ is not locally 2-transitive.

We apply Lemma \ref{lem:reduction-thick-tree} to $T'$ to obtain a thick tree $S$ such that
\begin{itemize}
\item $\rV(S) = \{v \in \rV(T') \amid \deg(v) \geq 3\}$,
\item $\rV(S)$ is $\Aut(T')$-invariant,
\item the restriction map $\Res: \Aut(T') \ra \Sym(\rV(S))$ induces an isomorphism of topological groups $\Aut(T') \cong \Aut(S)$,
\end{itemize}
Further, Lemma \ref{lem:reduction-thick-tree} says that since $H/K \grpaction{} T'$ acts minimally, $H/K \grpaction{} S$ has the same property.  Also if $H/K \grpaction{} T'$ is not locally 2-transitive, then $H/K \grpaction{} S$ has the same property.  Lemma \ref{lem:compactly-generated-cocompact} applies to show that $H/K$ acts cocompactly on $S$.  If $H/K \grpaction{} S$ is not type-preserving, we may replace $H$ by an index two subgroup of itself in order to guarantee also this property.  Note in particular, that $H/K \grpaction{} S$ remains minimal, since squares of hyperbolic elements are type-preserving.  This finishes the proof of the proposition.
\end{proof}

We are now ready to combine our results from Section \ref{sec:proper-actions} with Proposition \ref{prop:reduction} in order to prove our main theorem of this article.
\begin{proof}[Proof of Theorem \ref{thm:intro:non-amenability}]
  Let $T$ be a locally finite tree and $G \leqc \Aut(T)$ a closed non-amenable subgroup acting minimally on $T$.  Assume that $G$ does not act locally 2-transitively on $T$.  Since $G$ is not amenable, $\Aut(T)$ is not virtually abelian.  So Lemma \ref{lem:reduce-infinitely-many-ends} implies that $T$ has infinitely many ends.  Applying Proposition \ref{prop:reduction}, we find an non-amenable open subgroup $H \leq G$ a compact normal subgroup $K \unlhd H$ and a thick tree $S$ such that $H/K \leqc \Aut(S)$ acts minimally cocompactly type-preservingly and not locally 2-transitively on $S$.  In particular, $H/K$ is non-amenable.  Further $H/K \grpaction{} S$ is proper, since $H/K \leqc \Aut(S)$ is closed.  So the results of Section \ref{sec:proper-actions} apply to show that $\rL(H/K)$ is non-amenable.  Since Proposition~\ref{prop:quotient-is-corner} says that $\rL(H/K)$ is a corner of $\rL(H)$, also the latter von Neumann algebra is non-amenable.  Since $H \leq G$ is open, also $\rL(G)$ follows non-amenable by Proposition \ref{prop:amenability-open-subgroup}.  This finishes the proof of the theorem.
\end{proof}

\section{Applications to type I groups and to Burger--Mozes groups}
\label{sec:burger-mozes-groups}

In this section we will prove Theorems \ref{thm:intro:non-type-I} and \ref{thm:intro:characterisation-type-I-BM}.

\subsection{Type I groups}
\label{sec:applications-type-I}

\begin{definition}
  \label{def:type-I}
  Let $M$ be a von Neumann algebra.  We say that $M$ is a type ${\rm I}$ von Neumann algebra if for every projection $p \in M$ there is some $q \leq p$ (i.e. $pq = q$) such that $qMq$ is abelian. 

  A locally compact group $G$ is called a type ${\rm I}$ group, if every unitary representation of $G$ generates a type ${\rm I}$ von Neumann algebra.
\end{definition}

The following description of type ${\rm I}$ von Neumann algebras is well-known and provides the reader unfamiliar with this von Neumann algebraic notions with some orientation.
\begin{proposition}
  \label{prop:description-type-I}
  A von Neumann algebra $M$ is of type ${\rm I}$ if and only if there is a cardinal $\kappa$ and (possibly empty) measure spaces $X_\omega$, $\omega \leq \kappa$ such that $M \cong \bigoplus_{\omega \leq \kappa} \Linfty(X_\omega) \ot \bo(H_\omega)$, where $H_\omega$ is a Hilbert space with an orthonormal basis of cardinality $\omega$.
\end{proposition}

With this characterisation at hand, we see that every type ${\rm I}$ von Neumann algebra is amenable.
\begin{corollary}
  \label{cor:type-I-implies-amenable}
  Every type ${\rm I}$ von Neumann algebra is amenable.
\end{corollary}

We can now proceed to the proof of our main theorem's first application.
\begin{proof}[Proof of Theorem \ref{thm:intro:non-type-I}]
  This follows immediately from Theorem \ref{thm:intro:non-amenability} and Corollary \ref{cor:type-I-implies-amenable}.
\end{proof}

\subsection{Applications to Burger-Mozes groups}
\label{sec:applications-BM-groups}

The following property is the foundation of combinatorial considerations about type ${\rm I}$ groups acting on trees.
\begin{definition}
  \label{def:independence-property}
  Let $T$ be a locally finite tree.  If $e \in \rE(T)$ is an edge in $T$, then the graph $T$ without $e$ is a disjoint union of two trees, which we call the half trees emerging from $e$.

  A closed subgroup $G \leqc \Aut(T)$ has Tits' independence property if for all edges $e \in \rE(T)$ with half trees $\mathfrak h_1$, $\mathfrak h_2$ emerging from $e$ there is a decomposition $\Fix_G(e) = \Fix_G(\mathfrak h_1) \times \Fix_G(\mathfrak h_2)$. 
\end{definition}

An important class of examples enjoying Tits' independence property are Burger--Mozes groups.
\begin{definition}[Burger--Mozes \cite{burgermozes00-local-global}]
  \label{def:burger-mozes-groups}
  Let $n \geq 3$ and $T$ be the $n$-regular tree.  A legal colouring of $T$ is a map $l:\rE(T) \ra \{1, \dotsc, n\}$ such that $l(e) = l(\ol{e})$ for all $e \in \rE(T)$ and $l|_{\rE(v)}$ is a bijection for every $v \in \rV(T)$.  Given a legal colouring $l$ of $T$, we define the local action of $g \in \Aut(T)$ at $v \in \Aut(T)$ by
  \begin{equation*}
    \sigma(g, v) := l \circ g \circ l|_{\rE(v)}^{-1} \in \Sym(\{1, \dotsc, n\}) = \rS_n
    \eqstop
  \end{equation*}
  If $F \leq \rS_n$ is given, we define the Burger-Mozes groups by 
  \begin{align*}
    \rU(F)\phantom{^+} & := \{g \in \Aut(T) \amid \forall v \in \rV(T): \, \sigma(g,v) \in F\}
  \end{align*}
  and their type-preserving subgroups
  \begin{align*}
    \rU(F)^+ & := \rU(F) \cap \Aut(T)^+
    \eqstop
  \end{align*}
\end{definition}
Note that the definition of $\rU(F)$ and $\rU(F)^+$ a priori depends on the choice of a legal colouring.  However, the fact that a legal colouring is unique up precomposition with a tree automorphism shows that $\rU(F)$ and $\rU(F)^+$ are independent of this choice up to conjugation by a tree automorphism.  Since $\Aut(T)^+ \leq \Aut(T)$ has index 2, also $\rU(F)^+ \leq \rU(F)$ has index 2.  In this context, note that our definition of $\rU(F)^+$ as type-preserving part of $\rU(F)$ in general differs from the subgroup $\bigvee_{e \in \rE(T)} U(F)_e$ from BM, which could be trivial.  However, these two definitions agree in case $F$ is transitive and generated by point-stabilisers.

Thanks to Tits' independence property, $\rU(F)^+$ is abstractly simple, if $F$ is transitive and generated by point-stabilisers.  Burger--Mozes groups are an important class of examples in the theory of totally disconnected groups.  

Actually Burger-Mozes groups account for a large class of groups having Tits' independence property, as it is demonstrated by the following theorem.  Its statement did not yet appear in the literature, and we add it for the reader's convenience. The proof combines known results from Burger--Mozes \cite{burgermozes00-local-global} and Bank-Elder-Willis \cite{bankelderwillis14}.
\begin{theorem}
  \label{thm:characterisation-independence-property}
  Let $T$ be a locally finite tree and $G \leqc \Aut(T)$ a closed vertex and edge transitive group with Tits' independence property.  Let $F \leq \rS_n$ be permutation isomorphic with the image of $G_v$ in $\Sym(\rE(v))$.  Then $G = \rU(F)$ for a suitable colouring of $T$.
\end{theorem}
\begin{proof}
  Since $G$ is edge transitive, it is locally transitive.  So Proposition 3.2.2 of \cite{burgermozes00-local-global} applies to show that there is a suitable legal colouring of $T$ for which the inclusion $G \leq \rU(F)$ holds.  Theorem~5.4 of \cite{bankelderwillis14} says that 
  \begin{equation*}
    G
    =
    \{ g \in \Aut(T) \amid \forall v \in \rV(T) \, \exists h \in G: \, g|_{\rB_1(v)} = h|_{\rB_1(v)} \}
    =
    \rU(F)
    \eqstop
  \end{equation*}
This finishes the proof.
\end{proof}

The following result says that the type ${\rm I}$ conjecture holds for vertex transitive groups with Tits' independence property.  Note that non-compact boundary transitive groups are edge transitive.  So the previous theorem shows that Theorem \ref{thm:type-I} applies exactly to Burger-Mozes groups.
\begin{theorem}[Olshanskii \cite{olshanskii80},  Amann \cite{amann-thesis}, Ciobotaru {\cite[Theorem 3.5]{ciobotaru15}}]
  \label{thm:type-I}
  Let $T$ be a locally finite tree and $G \leqc \Aut(T)$ a closed subgroup acting transitively on vertices of $T$.  Assume that $G$ has Tits' independence property.  If $G$ acts transitively on the boundary $\partial T$, then $G$ is a type ${\rm I}$ group.
\end{theorem}

In order to formulate a converse to this theorem, which is the content of our Theorem \ref{thm:intro:characterisation-type-I-BM}, we need to characterise boundary transitivity of groups with Tits' independence property.  The next lemma is essentially contained in the ideas of Burger--Mozes' \cite[Lemma 3.1.1]{burgermozes00-local-global}.  It also appeared as Proposition 15 in \cite{amann-thesis}.  We claim no originality, but give a full proof for the convenience of the reader.
\begin{lemma}[Compare with Burger--Mozes \cite{burgermozes00-local-global}.  See also Proposition 15 in \cite{amann-thesis}]
  \label{lem:independence-locally-two-transitive}
  Let $T$ be a locally finite tree that is not a line nor a vertex and let $G \leqc \Aut(T)$ be a closed vertex transitive group with Tits' independence property.  Then $G$ is boundary transitive if and only if $G$ is locally 2-transitive.
\end{lemma}
\begin{proof}
  Since $G$ is vertex transitive, it is non-compact.  So Lemma 3.1.1 in \cite{burgermozes00-local-global} shows that if $G$ is transitive on the boundary, then $G$ is locally 2-transitive.  

In order to prove the converse we appeal to Lemma 3.1.1 \cite{burgermozes00-local-global} again and have to show that for every $v \in \rV(T)$ and every $n \in \NN$ the action of $G_v$ on $\partial \rB_n(v)$ is transitive.  Since $G \grpaction{} T$ is vertex transitive, $T$ is a homogeneous tree and its degree is at least three, since $T$ is not a line nor a vertex.  Let $x,y  \in \partial \rB_n(v)$ and let $r:[0,n] \ra T$, $s: [0,n] \ra T$ be the unique geodesics satisfying $\rmo(r) = \rmo(s) = v$, $\rmt(r) = x$ and $\rmt(s) = y$.  We inductively show the existence of $g_1, \dotsc, g_n \in G_v$ such that $(g_i r)(i) = s(i)$ for all $i \in \{1, \dotsc, n\}$.  Since $G$ is locally 2-transitive and $T$ is homogeneous of degree at least three, $G$ also acts locally transitively.  So there is some $g_1 \in G_v$ such that $g_1 r(1) = s(1)$.  Assume that $g_1, \dotsc, g_i$ have been constructed for $i < n$.  Let $\mathfrak h_1, \mathfrak h_2$ be the two half-trees emerging from the edge $e := (s(i - 1), s(i))$.  The notation can be fixed by assuming $s(i- 1) \in \mathfrak h_1$ and $s(i) \in \mathfrak h_2$.  Then $\mathfrak h_2$ contains all vertices adjacent to $s(i)$ that have distance $i + 1$ to $v$.  In particular, $s(i + 1), g_i r(i + 1) \in \mathfrak h_2$.  Since $G$ is locally 2-transitive and $|\rE(s(i))| \geq 3$, there is $h \in G_e$ satisfying $h(g_ir(i + 1)) = s(i+1)$.  Because $G$ has the independence property, we obtain the product decomposition $G_e = \Fix_G(\mathfrak h_1) \times \Fix_G(\mathfrak h_2)$ and can write $h = (h_1, h_2)$ with $h_1 \in \Fix_G(\mathfrak h_1)$ and $h_2 \in \Fix_G(\mathfrak h_2)$.  Then $h_1 g_i r(i + 1) = h_2^{-1} h g_i r(i + 1) = h_2^{-1} s(i+1) = s(i+1)$.  Further, $h_1 v = v$, since $v \in \rV(\mathfrak h_1)$.  We put $g_{i+1} := h_1 g_i$ and finish the induction.  Now the existence of $g_n$ with $g_n v = v$ and $g_n x = g_n r(n) = s(n) = y$ finishes the proof of the lemma.
\end{proof}

Let us reformulate Lemma \ref{lem:independence-locally-two-transitive} in terms of Burger-Mozes groups.
\begin{lemma}[Burger--Mozes {\cite[Section 3]{burgermozes00-local-global}}]
  \label{lem:bm-local-global}
  Let $F \leq \rS_n$ for $n \geq 3$ be given.  Then the following statements are equivalent.
  \begin{itemize}
  \item $\rU(F)$ is boundary transitive,
  \item $\rU(F)$ is locally 2-transitive,
  \item $F$ is 2-transitive.
  \end{itemize}
\end{lemma}

Combining Theorem \ref{thm:type-I}, Lemma \ref{lem:independence-locally-two-transitive} and Theorem \ref{thm:intro:non-type-I}, we obtain the characterisation of vertex transitive type ${\rm I}$ groups with the independence property, stated as Theorem \ref{thm:intro:characterisation-type-I-BM}.
\begin{proof}[Proof of Theorem \ref{thm:intro:characterisation-type-I-BM}]
  All statements of the theorem are obvious in case $T$ is a line or $n = 2$.

  Let $T$ be a locally finite tree and $G \leqc \Aut(T)$ a closed vertex transitive subgroup with Tits' independence property.  If $G$ is locally 2-transitive, then $G$ is boundary transitive by Lemma \ref{lem:independence-locally-two-transitive}.  So Theorem \ref{thm:type-I} says that $G$ is a type ${\rm I}$ group.  If $G$ is not locally 2-transitive, then $T$ has at least one vertex of degree 3.  So $T$ is not a line and it follows from vertex transitivity, minimality of $G \grpaction{} \partial T$ and Proposition \ref{thm:nebbia} that $G$ is not amenable.  So Theorem \ref{thm:intro:non-type-I} applies to show that $G$ is not a type ${\rm I}$ group.

It remains to prove the statement about Burger-Mozes groups. Since for every $F \leq \rS_n$ the closed subgroup $\rU(F)^+ \leq \rU(F)$ has index 2, it suffices to characterise when $\rU(F)$ is a type ${\rm I}$ group.  Now $\rU(F)$ is vertex transitive and has Tits' independence property.  So the first part of the statement says that $\rU(F)$ is a type ${\rm I}$ group if and only if it acts locally 2-transitively.  Now Lemma \ref{lem:bm-local-global} finishes the proof of the theorem.  
\end{proof}


\clearpage

\bibliographystyle{abbrv}
\bibliography{mybibliography}




\vspace{2em}
{\small \parbox[t]{200pt}
  {
    Cyril Houdayer \\
    Laboratoire de Math{\'e}matiques d’Orsay \\
    Universit{\'e} Paris-Sud \\
    CNRS \\
    Universit{\'e} Paris-Saclay \\
    F-91405 Orsay \\
    {\footnotesize cyril.houdayer@math.u-psud.fr}
  }
}
\hspace{15pt}
{\small \parbox[t]{200pt}
  {
    Sven Raum \\
    EPFL SB SMA\\
    Station 8 \\
    CH–1015 Lausanne \\
    Switzerland \\
    {\footnotesize sven.raum@epfl.ch}
  }
}

\end{document}